\documentclass[12pt]{amsart}

\usepackage{amscd}
\usepackage{verbatim}

\usepackage{caption}
\usepackage[labelformat=simple,labelfont={}]{subcaption}

\advance\hoffset by -0.5 truecm

\usepackage{amssymb}
\newtheorem{Theorem}{Theorem}[section]

\newtheorem{Lemma}[Theorem]{Lemma}

\advance\hoffset by -0.5 truecm
\usepackage{mathtools}
\usepackage{amsmath}
\usepackage{amssymb}
\usepackage{graphicx}
\usepackage{soul}
\usepackage{xcolor}
\usepackage{enumerate}
\newtheorem{Proposition}[Theorem]{Proposition}

\def\V{\mbox{Var}}

\def\R\re

\def\V{\bf V}

\def \re{{\mathbb R}}

\def \N{{\mathbb N}}

\def \0{\lambda_{0}}

\def \eps{\epsilon}

\begin{document}
\hyphenation{Ya-ma-be co-rres-pon-ding hy-po-the-sis iso-pe-ri-me-tric gen-er-at-ed man-i-folds me-tric differ-en-tial in-va-riants de-ve-lo-ped e-llip-tic}

\title[On K-Peak solutions for the Yamabe equation]{On K-peak solutions for the Yamabe equation on product manifolds}

\author[J. M. Ruiz]{Juan Miguel Ruiz$^{\dagger*}$}
 \thanks{$^\dagger$ ENES  UNAM. Departamento  de Matem\'aticas. Le\'on, Gto., M\'exico. mruiz@enes.unam.mx}

\author[A. V. Juarez]{Areli V\'azquez Ju\'arez$^{\ddagger }$}
\thanks{$^\ddagger$ ENES UNAM.  Departamento de Matem\'aticas. Le\'on, Gto., M\'exico. areli@enes.unam.mx.}
\thanks{* Corresponding author.}




\begin{abstract}  
Let $(M^n, g)$  and $(X^m, h)$ be   closed manifolds $m, n>2$, such that $(X, h)$ has constant positive scalar curvature. We consider the one parameter family of products $(M\times X, g+\eps^2 h)$, $\eps>0$.  We prove that if either the scalar curvature of $g$, $s_g$, is constant or a certain dimensional constant $\beta=0$, there is some function $\Phi:M\rightarrow \re$   that depends on $s_g$, the norm of the Ricci curvature of $g$ and the norm of the curvature tensor of $g$; such that if $\xi_0$ is a stable, isolated,  critical point of $\Phi$,   then for each $K\in\N$, there is some $\eps_0>0$ such that for every $\eps \in (0,\eps_0)$  the subcritical Yamabe equation $-\eps^2\Delta_g u+(1+{\bf{c}}\eps^2 s_g)u=u^q$  has a positive $K-$peak  solution, which concentrates around $\xi_0$. Here, ${\bf{c}}=\frac{N-2}{4(N-1)}$, $q=\frac{N+2}{N-2}$ and $N=n+m$. This provides solutions  for the Yamabe equation on Riemannian products $(M\times X, g+\eps^2 h)$ and  covers some remaining cases of previous results  which handle the case where $s_g$ has non-degenerate critical points and $\beta\neq0$. 
\end{abstract}

\maketitle



\section{Introduction}
Let $(M^n,g)$ be a closed manifold $n>2$. The Yamabe problem consists of finding a metric of constant scalar curvature in the conformal class of the  metric $g$. If one writes a conformal metric to $g$ as $\tilde g=u^{\frac{4}{n-2}}g$, where $u\in C^{\infty}(M)$ and $u>0$, the resulting metric $\tilde g$ will have constant scalar curvature $s_{\tilde g}=\lambda$ if and only if $u$ solves the Yamabe equation: 
\begin{equation}
	\label{Yamabe1}
	-a_n \Delta u+s_g u= \lambda u^{p_n-1}
\end{equation}
\noindent where $a_n=\frac{4(n-1)}{n-2}$ and $p_n=\frac{2n}{n-2}$. It is a classical result, solved in various steps by H. Yamabe \cite{Yamabe}, N.S. Trudinger \cite{Trudinger}, T. Aubin  \cite{Aubin}, and R. Schoen \cite{Schoen}, that such a positive, smooth, function $u$ always exists (see \cite{Parker}, \cite{Akutagawa}, for a detailed discussion of the Yamabe problem). It has been of interest to ask whether other metrics of constant scalar curvature exist in the same conformal class, this is, whether problem (\ref{Yamabe1}) has a unique solution. Uniqueness in this sense is known to be true for some special cases, for example, when the original metric has non-positive scalar curvature, by the maximum principle \cite{Aubin}, or when the original metric is an Einstein metric, not isometric to the sphere with the round metric, by the work of M. Obata \cite{Obata}. On the other hand, when the original metric has positive scalar curvature, a rich  variety of examples of different metrics with constant scalar curvature and unit volume  in the same conformal class have been found, through many techniques, see for instance \cite{Beta, Otoba, Otoba2, Petean2, Ruiz}. 

Let $(M^n, g)$  and $(X^m,h)$ be      closed manifolds $m, n>2$, such that $(X,h)$ has constant positive scalar curvature.  Consider the one parameter family of products 
$(M\times X, g+\eps^2 h)$. The Yamabe equation on this family is

\begin{equation}
	\label{Yamabe2}
	-a_N(\Delta_g +\Delta_{\eps^2 h})u + (s_g+\eps^{-2}s_h)u=u^{p_N-1}
\end{equation}

\noindent where $N=m+n$. This setting has been of interest recently, see for example \cite{DeLima, Henry, Petean}. One may renormalize   (\ref{Yamabe2}) so that $s_h=a_N$ and look only for solutions that depend solely on $M$, $u:M\rightarrow \re$.  In this case, equation (\ref{Yamabe2}) is equivalent to

\begin{equation}
	\label{Yamabe3}
	-\eps^{2}\Delta_g u + ({\bf{c}}s_g \eps^{2}+1)u=u^{p_N-1}
\end{equation}

\noindent where ${\bf{c}}=c_N= a_N^{-1}$. We model the solutions to (\ref{Yamabe3}) after the positive solution of a limit equation in $\re^n$, 
\begin{equation}
	\label{Kwong}
	-\Delta U +U=U^{p-1}
\end{equation}

\noindent where $p=p_N$, $2<p<\frac{2n}{n-2}$ and $n>2$. It is well known, by the work of Kwong, \cite{Kwong}, that there exists a unique, positive, spherically symmetric, function $U\in H^1(\re^n)$ such that $U$ is a solution to (\ref{Kwong}). Moreover, $U$ and its derivatives are exponentially decaying at infinity, this is

\begin{equation}
	\label{decaying}
		\lim_{|x|\rightarrow \infty} U(|x|)|x|^{\frac{n-1}{2}}e^{|x|}=c>0, \  \ \	\lim_{|x|\rightarrow \infty} U'(|x|)|x|^{\frac{n-1}{2}}e^{|x|}=-c.
\end{equation}

Through the article we will denote this solution by $U$, assuming $n$ and $p$ are clear from the context. 
Using this model solution in $\re^n$, for any $K\in \mathbb{N}$, we  construct a $K$-peak soultion on the manifold in the following way. For any $\eps>0$,  note that $U_{\eps}=U(\frac{x}{\eps})$ is a solution to
\begin{equation}
	\label{Kwongeps}
	-\eps\Delta U_{\eps} +U_{\eps}=U_{\eps}^{p-1}.
\end{equation}

Now, since $M$ is closed, we can fix $r_0>0$ such that for any $x\in M$, the exponential map  $exp_{x}:T_xM\rightarrow M$,  at $x$ restricted to $B(0,r_0)$, is a diffeomorphism. $B(0,r_0)$ denotes a ball in $\re^n$ centered at 0, with radius $r_0$ and we will denote  by $B_g(x,r)$ the geodesic ball in $M$ centered at $x$ with radius $r$.

For fixed $r<r_0$, we then define on $M$ the function $W_{\eps,\xi}:M\rightarrow \re$. Let $\chi_r$ be a smooth, radial cutoff function in $\re^n$,  $\chi_r:\re^n\rightarrow \re$ such that $\chi(z)=1$ if $z\in B(0,\frac{r}{2})$ and $\chi(z)=0$ if  $z\in \re^n \setminus B(0,r)$. Let

\begin{equation}
	\label{W}
	W_{\eps,\xi}(x)=\begin{cases}
		U_{\eps}(exp_{\xi}^{-1}(x)) \chi_r(exp_{\xi}^{-1}(x)) & \textit{if} \ \ x\in B_g(\xi,r)\\
		0 &  \textit{otherwise}
	\end{cases} 
\end{equation}

Note that $W_{\eps,\xi}(x)$ is an approximate solution to (\ref{Yamabe3}), which concentrates    around $\xi$ as $\eps\rightarrow 0$, by the exponential decay of $U$, eventually getting closer to an exact solution. Moreover, adding $K$ of these  peaks $W_{\eps,\xi}(x)$, centered at different $K$ strategic points $\xi_1, \xi_2, \dots, \xi_K\in M$,  we also get an approximate solution of (\ref{Yamabe3}): $\sum_{i=1}^{K}	W_{\eps,\xi_i}(x)$. Then, by adding some perturbations  we  will get an exact solution to equation (\ref{Yamabe3}).

For  $\mathbf{c}=c_N$, we define  
\begin{equation}
	\label{beta0}
	\beta:=  \mathbf{c} \int_{\re^n} U^2(z)-\frac{1}{n(n+2)}\int_{\re^n}|\nabla U(z)|^2|z|^2 dz,
\end{equation}

\noindent note that $\beta$ is a dimensional constant that depends only on $n$ and $m$. The precise values of $\beta$ for general $m, n$ are not known, to the best of our knowledge, although numerical computations for each case can be performed. Analytical computations have been done,  for example, for the case $m+n=4$, in \cite{DM}. On the other hand, in \cite{RR} the authors provided numerical computations for values of $m$ and $n$ such that $m+n\leq 9$. In all these cases $\beta<0$. 

If $\beta\neq 0$ and $s_g$ has isolated, stable, critical points, multipeak solutions like the ones described above were studied in \cite{RR}. In such case, the solutions concentrate around isolated critical points of the scalar curvature $s_g$. In the present article we work on some   remaining cases. In particular,  we show that, if either   $\beta=0$ or   $s_g$ is constant,  $K$-peak solutions concentrate around  an isolated, stable, critical point of the functional $\Phi:M\rightarrow\re$, given by

\begin{equation}
	\label{Phi}
	\Phi(\xi):=\frac{1}{120(n+2)}   \left(  -c_8\Delta_{g} s_g(\xi)+c_6||Ric_{\xi}||^2-3c_1||R_{\xi}||^2  \right)
	+c_7 s_g(\xi)^2+c_9 s_g(\xi)
\end{equation}

\noindent where $s_g$, $Ric_{\xi}$ and $R_{\xi}$ are the scalar curvature, Ricci curvature and tensor of curvature of $(M,g)$, respectively. Meanwhile $c_1,c_6, c_7, c_8$ and $c_9$ are constants given in equations (\ref{c1}) through (\ref{c9}) that depend on the dimensions $m, n$ and can be computed explicitly. Note that, for example, if $(M,g)$  is an Einstein manifold, then critical points of   $||R_{\xi}||^2$ would be critical points of $\Phi(\xi)$. 

The  solutions are $K$-peak in the sense that they are close to $\sum_{i=1}^{K}	W_{\eps,\xi_i}(x)$
with the norm $||\quad ||_{\eps}$, given by

 $$||u||_{\eps}^2:= \frac{1}{\eps^n}\left(\eps^2\int_M |\nabla_gu|^2d\mu_g + \int_M (\eps^2\mathbf{c}s_g+1)u^2d\mu_g\right)$$
 \begin{Theorem}
 	\label{Thm1}
Suppose  that $\beta=0$ or that $s_g$ is constant. Let $\xi_{0}$ be an isolated local $C^1$ stable critical point of the functional $\Phi(\xi)$. Then, for each positive integer $K$, there exists $\eps_0=\eps_0(K)>0$ such that for each $\eps\in (0,\eps_0)$ there are points $\xi_1^{\eps},\xi_2^{\eps},\dots, \xi_{K}^{\eps}\in M$, so that

$$\frac{d_g(\xi_i^{\eps},\xi_j^{\eps})}{\eps}\rightarrow \infty \textit{\ and \ }   d_g(\xi_0^{\eps},\xi_j^{\eps}) \rightarrow  0$$
 
 and a solution $u_{\eps}$ of problem (\ref{Yamabe3}),  so that
 
 $$||u_{\epsilon}-\sum_{i=1}^{K}W_{\eps, \xi_j^{\eps}}||_{\eps} \rightarrow 0$$
 	
 \end{Theorem}

Our approach uses the Lyapunov-Schmidt reduction techniques for subcritical elliptic PDEs on manifolds, like the ones applied in  \cite{Micheletti} and \cite{Deng}.  To deal with the case $s_g$ constant or $\beta=0$, we follow some ideas in \cite{DKM}, like expanding  the energy functional  to order 4, which gives place to the functional $\Phi(\xi)$. We also follow, in particular,  \cite{Dancer}  and \cite{RR}   which deal with multipeak solutions and  the subcritical Yamabe equation, respectively. 

The article is organized as follows. In section \ref{Notation} we introduce some notation and background, as well as a way to improve the approximate solution to higher order. In section \ref{ExistenceResult}  we prove Theorem \ref{Thm1},  assuming propositions \ref{Existence}, \ref{Expansion} and \ref{Critical}. Section \ref{Asymptotic} is dedicated to compute  the expansion of the energy functional, which is the content of proposition \ref{Expansion}. In section   \ref{FiniteDimensional} we prove proposition \ref{Existence}, regarding the finite dimensional reduction. Finally, in section \ref{TechnicalEstimates}, we provide some technical estimates and the proof of proposition \ref{Critical}. \\

\textbf{Acknowledgments.}
The authors were supported by grant UNAM-DGAPA-PAPIIT IN108824 and grant UNAM-DGAPA-PAPIIT IN108224, respectively.


\section{Approximate solutions}\label{Notation}

Critical points of the functional $E:H^1(\re^n)\rightarrow \re$ 

$$E(f)=\int_{\re^n} \left( \frac{1}{2} |\nabla f|^2+\frac{1}{2}f^2-\frac{1}{p}(f^+)^p\right)dx$$

\noindent are positive solutions to the limit equation (\ref{Kwong}). We denote by $S_0=\nabla E:H^1(\re^n)\rightarrow H^1(\re^n)$. Note that $S_0(U)=0$ and that $U$ is non degenerate in the sense that $Kernel(S_0'(U))$ is spanned by

$$\psi^i(x):=\frac{\partial U}{\partial x_i}(x)$$

\noindent for $i=1,2,\dots, n$. This implies that the solution $U_{\eps}$ to equation (\ref{Kwongeps}) is a critical point of the functional 

$$E_{\eps}(F)=\eps^{-n}\int_{\re^n} \left( \frac{\eps^2}{2} |\nabla f|^2+\frac{1}{2}f^2-\frac{1}{p}(f^+)^p\right)dx$$

We denote by ${S_0}_{\eps}=\nabla E_{\eps}$.  $U_{\eps}$ is non degenerate, $Kernel({S_0}_{\eps}'(U_{\eps}))$ is spanned by
$\psi_{\eps}^i(x):= \psi^i\left(\frac{x}{\eps}\right)$, for $i=1,2,\dots, n$. On $M$, we then may define

\begin{equation}
	\label{Z}
	Z^i_{\eps,\xi}(x):=\begin{cases}
		\psi^i_{\eps}(exp_{\xi}^{-1}(x)) \chi_r(exp_{\xi}^{-1}(x)) & \textit{if} \ \ x\in B_g(\xi,r)\\
		0 &  \textit{otherwise}
	\end{cases} 
\end{equation}

Let $H_{\eps}$ be the Hilbert space $H_g^1(M)$ with the inner product
$$\langle u,v \rangle_{\eps}:=\frac{1}{\eps^n}\left(\eps^2\int_M \nabla_gu\nabla_g v \ d\mu_g+\int_M(1+\eps^2 \mathbf{c}s_g)uv \  d\mu_g\right)$$

which induces the norm

$$||u||^2_{\eps}:= \frac{1}{\eps^n}\left(\eps^2\int_M |\nabla u|^2d\mu_g+\int_M (1+\eps^2 \mathbf{c}s_g)u^2d\mu_g\right)$$

Recall ${\bf c}=c_N=a_N^{-1}=\frac{N-2}{4(N-1)}$, $N=n+m$. Note  that for any function $f$, with $f_{\eps}(x)=f(\frac{x}{\eps})$, $||f_\eps||_{\eps}$ is independent of $\eps$.

For $\eps>0$ and $\bar \xi =(\xi_1, \xi_2, \dots \xi_{K})\in M^K$, let 

$$K_{\eps, \bar\xi}:=\text{span} \left\{Z^i_{\eps,\xi_j} | \ i=1,\dots,n\  \text{and} \  j=1,\dots, K\right\} $$

$$K_{\eps, \bar\xi}^{\perp}:=\left\{\phi \in H^1_{\eps} | \ \langle \phi, Z^i_{\eps,\xi_j}\rangle_{\eps}, i=1,\dots,n \  \text{and} \ j=1,\dots, K\right\} $$

Let $\Pi_{\eps,\xi}:H_{\eps}\rightarrow K_{\eps,\bar \xi}$ and $\Pi_{\eps,\xi}^{\perp}:H_{\eps}\rightarrow K_{\eps,\bar \xi}^{\perp}$ be orthogonal projections.


Let $L_{\eps}^p$ denote  the Banach space $L_g^p(M)$ with the norm $|u|_{p,\eps}:=\left(\frac{1}{\eps^n}\int_M |u|^p d\mu_g\right)^{1/p}$.
Since $2<p<\frac{2n}{n-2}$, it follows from the usual Sobolev embeddings, that there is a constant $c$ independent of $\eps$ such that 
\begin{equation}
\label{upe}
|u|_{p,\eps}\leq c ||u||_{\eps}, \text{\ for any\ }  u\in H_{\eps}
\end{equation}

In particular, the embedding $H_{\eps} \xhookrightarrow{} L_{\eps}^p$ is a compact continuous map. We denote  $p':=\frac{p}{p-1}$. We identify the dual space $L^p_{\eps}$ with $L^{p'}_{\eps}$, with the pairing

$$\langle \varphi, \psi \rangle =\frac{1}{\eps^n}\int_M \varphi \psi$$

\noindent for $\varphi \in   L^p_{\eps}$,  $\psi \in L^{p'}_{\eps}$. The adjoint operator $i^*_{\eps}:L_{\eps}^{p'}  \rightarrow H_{\eps}$
is a continuous map so that

\begin{equation}
u=i^*_{\eps}(v) \iff \langle i^*_{\eps}(v) ,  \varphi \rangle_{\eps}=\frac{1}{\eps^n}\int_M v\varphi, \varphi \in H_{\eps} \iff 
	 -\eps^2 \Delta_g u+(1+{\bf c}s_g\eps^2)u=v \text{\ (weakly) on\ } M\\
\end{equation}

For the same constant as in (\ref{upe}), for any $v\in L^{p'}_{\eps}$,

\begin{equation}
	\label{25}
	||i^*_{\eps}(v)||_{\eps}\leq c |v|_{p',\eps}
\end{equation}

Let $f(u):=(u^+)^{p-1}$. Note that 
\begin{equation}
	S_{\eps}(u)=u -i^*_{\eps}(f(u)), u\in H_{\eps}
\end{equation}

so that problem (\ref{Yamabe3}) can be rewritten as
\begin{equation}
	\label{ui*}
	u=i^*_{\eps}(f(u)), u\in H_{\eps}
\end{equation}


The following  will be useful. They are computed, for example, in \cite{Gray}. 

\begin{Lemma}
	\label{expansions}
		 The following expansions hold.\\
	\begin{enumerate}[i)]
		\item  In a normal coordinates neighborhood of $\xi \in M$,
	\begin{equation}
		\label{Lapla}
		\Delta_g u=\Delta_{\re^n} u -\frac{\eps^2}{3} R_{kijl} z_k z_l \partial_{ij}^2 u +\frac{2\eps^2}{3} R_{kssj}z_k \partial_j u+o(\eps^2 |z|^2)
	\end{equation}

	\begin{equation}
		g_{ij}=\delta_{ij}-\frac{\eps^2}{3} R_{kijl}z_kz_l-\frac{\eps^3}{6}\nabla_m R_{kijl}z_kz_lz_m-\eps^4\left(\frac{1}{20}\nabla_{pq}R_{kijl}+\frac{2}{45}R_{kilr}R_{pjqr}\right) z_kz_lz_pz_q
	\end{equation}
	$$+\frac{\eps^4}{9}R_{kisl}R_{psjq}z_kz_lz_pz_q+o(\eps^4|z|^4)$$

		\begin{equation}
\sqrt{det(g)}=1-\eps^2\frac{1}{6}R_{kl}z_kz_l-\frac{\eps^3}{12}\nabla_m R_{kl}z_kz_lz_m
	\end{equation}
$$-\eps^4\left(\frac{1}{40}\nabla_{pq}R_{kl}+\frac{1}{180}R_{kilr}R_{piqr}-\frac{1}{72}R_{kl}R_{pq}\right)z_kz_lz_pz_q	+o(\eps^4|z|^4)$$

\item 	Let $w(z)=u(\exp_{\xi}(z))$, $z\in B(0,r)\subset\re^n$, $\xi \in M$,
		\begin{equation}
	\label{Gamma}
	\Delta_{\re^n} w=\Delta_{g_{\xi}}w +(g^{ij}_{\xi}-\delta_{ij})\partial^2_{ij}w-\eps^2 g^{ij}_{\xi_i}\Gamma_{ij}^k\partial_k w+o(\eps^2)
		\end{equation}
	
	\noindent where $R_{ijkl}$ and $R_{kl}$ denote the curvature tensor and the Ricci curvature.
\end{enumerate}

\end{Lemma}

\noindent We will use a second order improvement of the approximate solution $W_{\eps, \xi}$ of the form
\begin{equation}
	\label{u1}
	{u_1}_{\eps, \xi}=W_{\eps, \xi}+\eps^2 V_{\eps, \xi}
\end{equation}

\noindent where $V_{\eps, \xi}$ is a function  on $M$, such  that the second order terms of $\eps$ vanish for 
$$-\eps^2\Delta_g 	{u_1}_{\eps, \xi}+(1+\eps^2{\bf c}s_g)	{u_1}_{\eps, \xi}-|	{u_1}_{\eps, \xi}|^{p-1}$$
 
More precisely, let  $V_{\eps,\xi}:M\rightarrow\re$, with

\begin{equation}
	\label{V}
	V_{\eps,\xi}(x):=\begin{cases}
		V_{\eps}(exp_{\xi}^{-1}(x)) \chi_r(exp_{\xi}^{-1}(x)) & \textit{if} \ \ x\in B_g(\xi,r)\\
		0 &  \textit{otherwise}
	\end{cases} 
\end{equation}

\noindent where $V_{\eps}(z):= V(\frac{z}{\eps})$ and $V:\re^n\rightarrow \re$ is the unique bounded solution of the problem

\begin{equation}
	\label{Lov}
	L_0v=-\frac{1}{3} R_{kijl}z_k z_l\partial_{ij}^2 U+\frac{2}{3} R_{kssj} z_k \partial_j U-{\bf c}s_g U
\end{equation}

\noindent where $L_0$ is the linear operator defined by,
\begin{equation}
	\label{Lov1}
	L_0v:=-\Delta_{\re^n}v+v-(p-1)|U|^{p-2}v
\end{equation}	

Since $U$ is radial,

$$
	\partial_{i}U=z_i \frac{U'}{|z|}\textnormal{\  \   \ and \  \   \ }\partial_{ij}^2U=\delta_{ij} \frac{U''}{|z|^2}+z_iz_j \frac{U'}{|z|}-z_iz_j \frac{U'}{|z|^3}
$$
then $$R_{kijl} z_kz_l \partial_{ij}^2U=R_{kiil} z_kz_l \frac{U'}{|z|}+ R_{kijl} z_kz_lz_iz_j  \frac{U''}{|z|^2}- R_{kijl} z_kz_lz_iz_j\frac{U'}{|z|^3}$$

\noindent with the second and third terms being zero, by the antisymmetry of the curvature tensor. Thus, we have,

 $$-\frac{1}{3} R_{kijl}z_k z_l\partial_{ij}^2 U+ \frac{2}{3} R_{kssj} z_k \partial_j U=\frac{1}{3}  R_{kl} z_k \ \partial_l U$$

 \noindent this implies that eq. (\ref{Lov}) is equivalent to
\begin{equation}
	\label{LovU}
	L_0v=\frac{1}{3}  R_{kl} z_k \ \partial_l U-{\bf c}s_g U
\end{equation}

Since $L_0$ is linear, one can look for a solution $V=V_1+V_2$ such that 
\begin{equation}
	\label{LovU1}
	L_0V_1=\frac{1}{3}  R_{kl} \ z_k \ \partial_l U 
\end{equation}
and
\begin{equation}
	\label{LovU2}
	L_0V_2=-{\bf c}s_g U
\end{equation}

\noindent  Note that (\ref{LovU1}) is always solvable since the right-hand side is orthogonal to the Kernel of $L_0$ (which is spanned by $\partial_i U$, $i=1, 2, \dots, n$). Moreover, since $U$ is radial, then $ \partial_l U=\frac{U'}{|z|} z_l$, hence  eq. (\ref{LovU1}) is equivalent to

\begin{equation}
	\label{Lov2}
	L_0V_1=\frac{1}{3} R_{kl} \frac{U'}{|z|}z_k z_l
\end{equation}

By direct computation,  for a radial function $\varphi=\varphi(|z|)$ on $\re^n$,

$$L(\varphi z_kz_l)=\left(\varphi''+(n+3)\frac{\varphi'}{|z|}+\varphi-(p-1)U^{p-2}\varphi \right)z_kz_l$$

Then, one can look for a radial function $\psi=\psi(|z|)$, such that 
\begin{equation}
	\label{Psiog}
L(\psi z_kz_l)=\frac{U'}{|z|}z_kz_l
\end{equation}

\noindent since $\frac{U'}{|z|}$ is radial. In this way, the solution to   (\ref{Lov2})  will be given by

\begin{equation}
	\label{Vestimation}
	V_1=	\frac{1}{3} R_{kl} \ \psi \ z_k z_l
\end{equation}

Moreover, by the form of the operator $L_0$ (eq. (\ref{Lov1})) and of the right hand side of (\ref{Lov2}), $V_1$ and its derivatives are exponentially decaying at infinity.

On the other hand, by direct computation 
\begin{equation}
	\label{TrueLaplacian}
	L_0(  U'|z| )= -2\Delta U \textnormal{\ \ \ and \ \ \ }
	L_0( U)=(2-p) U^{p-1}
\end{equation}
Hence,
\begin{equation}
	\label{Vestimation2}
	V_2= {\bf c}s_g\left( \frac{1}{2} U'|z|-\frac{1}{2-p}U\right)
\end{equation} 
satisfies (\ref{LovU2}). Of course, this implies that $V=V_1+V_2$  satisfies (\ref{LovU}) and, moreover, that $V$ and its derivatives are exponentially decaying at infinity.

For $\xi\in M$ and $\eps>0$, let

\begin{equation}
Y_{\eps, \xi}	:= W_{\eps,\xi} +\eps^2 V_{\eps,\xi}
\end{equation}

and, for $\xi_1, \xi_2, \dots, \xi_K \in M$, with $\bar \xi= (\xi_1, \xi_2, \dots, \xi_K)$,

\begin{equation}
	\label{YY}
	Y_{\eps, \bar \xi}	=\sum_{i=1}^{K} \left(W_{\eps,\xi_i} +\eps^2 V_{\eps,\xi_i}\right)
\end{equation}

Let $\xi_0\in M$ an isolated, stable, $C^1$ critical point of the functional $\Phi$ (eq. (\ref{Phi})). Let $K\geq 1$ be an integer. Given $\rho>0$, $\eps>0$, let

\begin{equation}
	\label{Dk}
	D_{\eps,\rho}^K:=\{\bar \xi \in M^K | \  d_g(\xi_0,\xi_i)<\rho, \  i=1,2,\dots, K,  \ \sum_{i\neq j}^{K}\left(U_{\eps}(\exp_{\xi_i}^{-1}\xi_j)\right)<\eps^4 \}
		\end{equation}

Note that $\bar\xi_{\eps}\in 	D_{\eps,\rho}^K$ implies, for $i\neq j$,

\begin{equation}
	\label{fraction}
	\lim_{\eps \rightarrow 0} \frac{d_g({\xi_{\eps}}_i,{\xi_{\eps}}_j)}{\eps}=+\infty
\end{equation}

This follows from the fact that $$U_{\eps}(\exp_{\xi_i}^{-1}\xi_j)=U\left(\frac{\exp_{\xi_i}^{-1}\xi_j}{\eps}\right)=U\left(\frac{d_g({\xi_{\eps}}_i,{\xi_{\eps}}_j)}{\eps}\right)<\eps^4 $$

\noindent and $U$ is a radial, positive, exponentially decreasing function. Moreover, it follows from (\ref{decaying}) that, for any $\delta>0$ we have

\begin{equation}
	\label{delta1}
	\lim_{\eps \rightarrow 0} \frac{1}{\eps^4}e^{-(1+\delta)\frac{d_g({\xi_{\eps}}_i,{\xi_{\eps}}_j)}{\eps}}=0
\end{equation}

\noindent since, otherwise, if for some $a>0$, $
	\lim_{\eps \rightarrow 0}  e^{-(1+\delta)\frac{d_g({\xi_{\eps}}_i,{\xi_{\eps}}_j)}{\eps}}>a\eps^4$. Then, for some small and positive values of $\epsilon$ ($\eps\rightarrow0$), 	$$	\lim_{\eps \rightarrow 0}  e^{-\frac{d_g({\xi_{\eps}}_i,{\xi_{\eps}}_j)}{\eps}}>a\eps^4   e^{\delta\frac{d_g({\xi_{\eps}}_i,{\xi_{\eps}}_j)}{\eps}}.$$

From this, applying (\ref{decaying}) with $|x|= \frac{d_g({\xi_{\eps}}_i,{\xi_{\eps}}_j)}{\eps}$, and since $ U\left(\frac{d_g({\xi_{\eps}}_i,{\xi_{\eps}}_j)}{\eps}\right)<\eps^4 $:

\begin{equation}
 \eps^4\left(\frac{d_g({\xi_{\eps}}_i,{\xi_{\eps}}_j)}{\eps}\right)^{\frac{n-1}{2}}>c e^{-\frac{d_g({\xi_{\eps}}_i,{\xi_{\eps}}_j)}{\eps}}>c a\eps^4   e^{\delta\frac{d_g({\xi_{\eps}}_i,{\xi_{\eps}}_j)}{\eps}},
\end{equation}

\noindent which gives a contradiction as $\eps\rightarrow 0$, since one side has polynomial growth and the other one exponential growth. In a similar way,   for $\delta>0$ small, also

\begin{equation}
	\label{delta2}
	\lim_{\eps \rightarrow 0} \frac{1}{\eps^4}e^{-(1-\delta)\frac{d_g({\xi_{\eps}}_i,{\xi_{\eps}}_j)}{\eps}}=0
\end{equation}

\noindent since, otherwise, if for some $b>0$, $	\lim_{\eps \rightarrow 0}  e^{-(1-\delta)\frac{d_g({\xi_{\eps}}_i,{\xi_{\eps}}_j)}{\eps}}>b\eps^4$, for some small and positive values of $\epsilon$, as $\eps\rightarrow0$. And then, for those values,

\begin{equation}
	\lim_{\eps \rightarrow 0}  e^{-\frac{1}{2}\frac{d_g({\xi_{\eps}}_i,{\xi_{\eps}}_j)}{\eps}}>b\eps^4   e^{(\frac{1}{2}-\delta)\frac{d_g({\xi_{\eps}}_i,{\xi_{\eps}}_j)}{\eps}}.
\end{equation}

From this, applying (\ref{decaying}) with $|x|=\frac{1}{2} \frac{d_g({\xi_{\eps}}_i,{\xi_{\eps}}_j)}{\eps}$, and since $U$ is decreasing, $$U\left(\frac{1}{2}\frac{d_g({\xi_{\eps}}_i,{\xi_{\eps}}_j)}{\eps}\right)\leq  U\left(\frac{d_g({\xi_{\eps}}_i,{\xi_{\eps}}_j)}{\eps}\right)<\eps^4, $$

\noindent hence
\begin{equation}
	\eps^4\left(\frac{1}{2} \frac{d_g({\xi_{\eps}}_i,{\xi_{\eps}}_j)}{\eps}\right)^{\frac{n-1}{2}}>c e^{-\frac{1}{2} \frac{d_g({\xi_{\eps}}_i,{\xi_{\eps}}_j)}{\eps}}>c b\eps^4   e^{(1-2\delta)\left(\frac{1}{2}\frac{d_g({\xi_{\eps}}_i,{\xi_{\eps}}_j)}{\eps}\right)},
\end{equation}

\noindent which gives a contradiction, for small $\delta$, as $\eps\rightarrow 0$, since one side has polynomial growth and the other, exponential growth. Equations (\ref{delta1}) and (\ref{delta2}) imply that $e^{-(1+\delta)\frac{d_g({\xi_{\eps}}_i,{\xi_{\eps}}_j)}{\eps}}$ and $e^{-(1-\delta)\frac{d_g({\xi_{\eps}}_i,{\xi_{\eps}}_j)}{\eps}}$ are $o(\eps^4)$ for ${\xi_{\eps}}_i$, ${\xi_{\eps}}_j \in \bar \xi_{\eps}$, $i\neq j$, with $\bar \xi_{\eps}\in D^k_{\eps,\rho}$.

We will look for a solution to (\ref{Yamabe3}) of the form
\begin{equation}
	\label{uepsilon1}
	u_{\eps}=Y_{\eps, \bar\xi}+\phi_{\eps,\bar\xi}
\end{equation}

\noindent where   $Y_{\eps, \bar\xi}$ is defined in (\ref{YY}), $\bar \xi=\bar \xi_{\eps}\in D^k_{\eps,\rho}$ and the rest term $\phi_{\eps,\bar\xi}\in K^{\perp}_{\eps,\bar \xi}$.

Positive solutions to problem (\ref{Yamabe3}) are critical points of the functional $J_{\eps}:H^1(M)\rightarrow \re$, given by

\begin{equation}
	\label{Jepsu}
	J_{\eps}(u)=\frac{1}{\eps^n} \int_M\left(\frac{\eps^2}{2}|\nabla u|^2+\frac{1}{2}(1+{\bf c}s_g\eps^2)u^2-\frac{1}{p_N}(u^+)^{p_N}\right)d\mu_g
\end{equation}
\noindent with $u^+(x)=\max\{u(x),0\}$.

Let $S_{\eps}=\nabla J_{\eps}: H_{\eps}\rightarrow H_{\eps}$. To solve problem (\ref{Yamabe3}) we must solve
$S_{\eps}(u_{\eps})=0$. Or in an equivalent way, we must solve the system,

\begin{equation}
	\label{perp}
	\Pi_{\eps,\xi}^{\perp}\left\{ S_{\eps}(Y_{\eps,\bar \xi}+\phi_{\eps,\bar\xi})\right\}	=0.
\end{equation}

\begin{equation}
	\label{proj}
	\Pi_{\eps,\xi} \left\{ S_{\eps}(Y_{\eps,\bar \xi}+\phi_{\eps,\bar\xi})\right\}	=0.
\end{equation}

Note that (\ref{proj}) is a  finite dimensional problem. 


\section{Existence result (proof of Theorem 1.1)}\label{ExistenceResult}

We will use the following to prove Theorem 1.1. The proof of this proposition is postponed  to section \ref{FiniteDimensional}.

\begin{Proposition}
	\label{Existence}
	There exist $\rho_0>0$, $\eps_0>0$, $c>0$  and $\sigma>0$ such that for any $\rho \in(0,\rho_0)$, $\eps\in(0,\eps_0)$ and $\bar \xi \in D_{\eps, \rho}^{K}$ there is a unique $\phi_{\eps,\bar \xi}=\phi(\eps,\bar \xi)\in K_{\eps,\bar \xi}^{\perp}$, which solves equation (\ref{perp}) and satisfies
\begin{equation}
	\label{barfi}
	|| \phi_{\eps,\bar \xi} ||_{\eps}\leq c\left(\eps^3+\sum_{i\neq j}e^{-\frac{(1+\sigma)}{2}\frac{d_g(\xi_i,\xi_j)}{\eps}}\right)
\end{equation}

Moreover, $\bar \xi\rightarrow \phi_{\eps,\bar \xi}$ is a $C^1$ map. 
\end{Proposition}

 For $\bar \xi \in M^K$, recall that $Y_{\eps, \bar\xi}=\sum_{i=1}^K (W_{\eps,\xi_i}+\eps^2 V_{\eps,\xi_i})$, $W_{\eps,\xi_i}$ and $V_{\eps,\xi_i}$ being defined in (\ref{W}) and (\ref{V}), respectively.  We define the functional $\bar J_{\eps}:D_{\eps,\rho}^K\subset M^K\rightarrow \re$, by

 \begin{equation}
 	\label{Jeps}
 	\bar J_{\eps}(\bar \xi):= J_{\eps}\left(Y_{\eps, \bar\xi}+\phi_{\eps,\bar \xi} \right)
 \end{equation}
 
 \noindent where $\phi_{\eps,\bar \xi}$ is given by Proposition \ref{Existence}. The following gives the expansion of this energy functional. Its proof  is presented in section \ref{Asymptotic}.

 \begin{Proposition}
 	\label{Expansion}
 	It holds

 \begin{equation}
 	\label{exptotal}
 	\bar{J}_{\eps}(\bar \xi)= K \alpha+  \eps^2 \frac{1}{2}\beta \sum_{i=1}^{K}s_g(\xi_i)+\eps^4  \sum_{i=1}^{K}\Phi(\xi_i)-\frac{1}{2} \sum_{i,j=1, i\neq j}^K \gamma_{ij} U\left(\frac{exp_{\xi_i}^{-1} \xi_j}{\eps}\right)+o(\eps^4)
 \end{equation}

$C^0$ uniformly with respect to $\bar \xi$ in compact sets of $D_{\eps,\rho}^K$ as $\eps$ goes to zero. 

Here,

 \begin{equation}
 	\label{alpha}
 	\alpha:= \frac{1}{2}\int_{\re^n}|\nabla U(z)|^2dz+\frac{1}{2}\int_{\re^n} U^2(z)dz-\frac{1}{p}\int_{\re^n} U^p(z)dz
 \end{equation}
 
 \begin{equation}
 	\label{beta}
 	\beta:= \mathbf{c} \int_{\re^n}U^2(z)dz-\frac{1}{n(n+2)}\int_{\re^n}|\nabla U(z)|^2|z|^2dz
 \end{equation}
 
  \begin{equation}
 	\label{gamma}
 	\gamma_{ij}:=  \int_{\re^n}U^{p-1}(z) e^{\langle b_{ij},z\rangle}dz
 \end{equation}
 
   \begin{equation}
 	\label{bij}
 	b_{ij}:= \lim_{\eps \rightarrow 0}  \frac{exp_{\xi_i}^{-1} \xi_j}{|exp_{\xi_i}^{-1} \xi_j|}
 \end{equation}
 
\begin{equation}
	\label{Phi2}
\Phi(\xi):=\frac{1}{120(n+2)}   \left(  -c_8\Delta_{g} s_g(\xi)+c_6||Ric_{\xi}||^2-3c_1||R_{\xi}||^2  \right)
+c_7 s_g(\xi)^2+c_9 s_g(\xi)
\end{equation}

  \noindent where  $c_1,c_6, c_7, c_8$ and $c_9$ are constants that depend on the dimensions $m, n$, and are given in equations (\ref{c1}) through (\ref{c9}).

\end{Proposition}

 Note next that critical points of $\bar J_{\eps}$ are solutions to problem (\ref{proj}), by the following.

	\begin{Proposition}
		\label{Critical}
		If $\bar\xi$ is an isolated, stable, $C^1$ critical point of $\bar J_{\eps}$, then the function $Y_{\eps,\bar\xi}+\phi_{\eps,\bar\xi}$ is a solution to equation (\ref{proj}), or, equivalently, to problem (\ref{Yamabe3}).	
	\end{Proposition}
	
	We present the proof of this proposition in section \ref{TechnicalEstimates}.  We now prove Theorem \ref{Thm1}.
 
 \begin{proof}(of Theorem \ref{Thm1})
 	
 	Let $\bar \xi_{\eps}\in \overline{D^K_{\eps,\rho}}$ be a solution to the maximizing problem
 	
 	\begin{equation}
 		\label{max}
 		\bar J_{\eps}(\bar \xi_{\eps})= \max \{\bar J_{\eps}(\bar \xi): \bar\xi\in \overline{D^K_{\eps,\rho}} \}
 	\end{equation}

 	We claim that, in fact, $\bar \xi_{\eps}\in D^K_{\eps,\rho}$. 
 	
 	Note that since $\xi_0$ of the hypothesis is a stable, critical, isolated point of $\Phi$, we may assume that is a local maximum of $\Phi$. If it were a minimum, the proof would be the same, nevertheless, we would look for a solution to a minimizing problem in eq. (\ref{max}).

 	We first note that by Proposition \ref{Expansion}, eq. (\ref{exptotal}), for $\bar \xi_{\eps}$:

 	 \begin{equation}
 		\label{exptotaleps}
 		\bar{J}_{\eps}(\bar \xi_{\eps})= K \alpha+  \eps^2 \frac{1}{2}\beta \sum_{i=1}^{K}s_g({\xi_{\eps}}_i)+\eps^4  \sum_{i=1}^{K}\Phi({\xi_{\eps}}_i)-\frac{1}{2} \sum_{i,j=1, i\neq j}^K \gamma_{ij} U\left(\frac{exp_{{\xi_{\eps}}_i}^{-1} {\xi_{\eps}}_j}{\eps}\right) +o(\eps^4)
 	\end{equation}
 	
 	We now construct a particular $ \bar \eta_{\eps}\in D^K_{\eps,\rho}$, so that we can make some estimates. Let $\bar\eta_{\eps}=(\eta_1,\eta_2,\dots, \eta_K)$, with $\eta_i=\eta_i(\eps)=\exp_{\xi_0}(\sqrt{\eps}e_i)$, $i\in \{1,2\dots,K\}$, and $e_1, e_2,\dots, e_K\in \re^n$, $e_i\neq e_j$, for $i\neq j$.
 	
 	By direct computation, $\bar \eta_{\eps}$ satisfies
 	
 	\begin{enumerate}[i)]
 		\item $d_g(\xi_0,\eta_i)=\sqrt{\eps}|e_i|$.
 		\item   $d_g(\eta_i,\eta_j)=|\exp_{\eta_i}^{-1} \eta_j|=\sqrt{\eps}(|e_i-e_j|+o(1)).$

	\end{enumerate}
 	
Moreover, 
\begin{equation}
	\label{U4}
	U\left(\frac{\exp_{\eta_i}^{-1}\eta_j}{\eps}\right)=U\left(\frac{d_g(\eta_i,\eta_j)}{\eps}\right)=U\left(\frac{\sqrt{\eps}(|e_i-e_j|+o(1))}{\eps}\right)=o(\eps^4)
\end{equation} 	
 	
 	\noindent for $\eps$ small enough, since $U$ is radially symetric and of exponential decay. This implies $  \bar \eta_{\eps}\in D^K_{\eps,\rho}$ holds, if $\eps$ is small enough. 	By (\ref{U4}) and (\ref{exptotal})
 	
 	 \begin{equation}
 		\label{exptotaleta}
 		\bar{J}_{\eps}(\bar \eta_{\eps})= K \alpha+  \eps^2 \frac{1}{2}\beta \sum_{i=1}^{K}s_g(\eta_i)+\eps^4  \sum_{i=1}^{K}\Phi(\eta_i)+o(\eps^4)
 	\end{equation}
 	
 	Note that, since $\xi_0$ is a critical point,
 	\begin{equation}
 		\label{eta}
 		\Phi(\eta_i)=\Phi(\xi_0)+\frac{1}{2}\Phi''(\xi_0) d(\xi_0,\eta_i)^2+o(\eps)=\Phi(\xi_0)+\eps \Phi''(\xi_0) |e_i|^2+o(\eps)
 	\end{equation}
 	$$=\Phi(\xi_0)+o(1)$$

 	Taking into account (\ref{eta}) and that either $\beta=0$ or $s_g$ is constant, (\ref{exptotaleta}) turns into,

 	 	 \begin{equation}
 		\label{exptotaleta2}
 		\bar{J}_{\eps}(\bar \eta_{\eps})= K \alpha+  \eps^2 \frac{1}{2}\beta K s_g(\xi_0)+\eps^4  K \Phi(\xi_0)+o(\eps^4)
 	\end{equation}

 	Now, since $\bar \xi_{\eps}$ is a maximum of $\bar J_{\eps}$ in  $\overline{D^K_{\eps,\rho}}$,
 	\begin{equation}
 		\bar{J}_{\eps}(\bar \xi_{\eps})\geq \bar{J}_{\eps}(\bar \eta_{\eps})
 	\end{equation}
 	
 	Using (\ref{exptotaleps}) on the left hand side and (\ref{exptotaleta2}) on the right hand side, together with the assumption that $\beta=0$ or $s_g$ is constant, we get

 	\begin{equation}
 		K \alpha+  \eps^2 \frac{1}{2}\beta K s_g(\xi_0)-\frac{1}{2} \sum_{\substack{i,j=1\\ i\neq j}}^K \gamma_{ij} U\left(\frac{exp_{{\xi_{\eps}}_i}^{-1} {\xi_{\eps}}_j}{\eps}\right)+\eps^4  \sum_{i=1}^{K}\Phi({\xi_{\eps}}_i)+o(\eps^4)
 	\end{equation}
 	
 	$$\geq K \alpha+  \eps^2 \frac{1}{2}\beta K s_g(\xi_0)+\eps^4  K \Phi(\xi_0)+o(\eps^4)$$
\noindent this is

 	 	\begin{equation}
 	 		\label{comparison}
 	\frac{1}{2} \sum_{i,j=1, i\neq j}^K \gamma_{ij} U\left(\frac{exp_{{\xi_{\eps}}_i}^{-1} {\xi_{\eps}}_j}{\eps}\right)+\eps^4\left(  K \Phi(\xi_0)-  \sum_{i=1}^{K}\Phi({\xi_{\eps}}_i)\right)\leq o(\eps^4)
 	\end{equation}
 	
 	We now fix $\rho$ small enough so that $\xi_0$ is the only maximum of $\Phi$ in $B_g(\xi_0, \rho)$. Hence, each term of (\ref{comparison}) is non-negative, and therefore, bounded from above by $o(\eps^4)$. Thus
 
 		\begin{equation}
 		\label{null}
0\leq  \eps^4\left(  K \Phi(\xi_0)-  \sum_{i=1}^{K}\Phi({\xi_{\eps}}_i)\right)\leq o(\eps^4)
 	\end{equation}
 	
 As a consequence 	$0\leq K \Phi(\xi_0)-  \sum_{i=1}^{K}\Phi({\xi_{\eps}}_i)=o(1)$. It follows that $\lim_{\eps\rightarrow0} \Phi({\xi_{\eps}}_i)=\Phi(\xi_0)$. Hence, $\eps$ small enough implies
 
 \begin{equation}
 	\label{rho}
 	d_g({\xi_{\eps}}_i,\xi_0)<\rho.
 \end{equation}

 On the other hand,  also by (\ref{comparison}), and recalling that $\gamma_{ij}=\int_{\re^n} U^{p-1}(z)e^{\langle b_{ij},z \rangle}dz$ and $|b_{ij}|=1$, by (\ref{bij}), for all $i,j\leq K$
 	
 	\begin{equation}
 		o(\eps^4)\geq \sum_{\substack{i,j=1\\ i\neq j}}^K \gamma_{ij} U\left(\frac{exp_{{\xi_{\eps}}_i}^{-1} {\xi_{\eps}}_j}{\eps}\right)\geq \gamma \sum_{\substack{i,j=1\\ i\neq j}}^K  U\left(\frac{exp_{{\xi_{\eps}}_i}^{-1} {\xi_{\eps}}_j}{\eps}\right)
 	\end{equation}
 	
 	\noindent where $\gamma:=\min \left\{\int_{\re^n} U^{p-1}(z)e^{\langle b_{ij},z \rangle}dz | b_{ij}\in \re^n, |b_{ij}|=1\right\}>0$. 	Hence, for $\eps$ small enough
 	
 	\begin{equation}
 		\label{U4gamma}
 	\sum_{\substack{i,j=1\\ i\neq j}}^K  U\left(\frac{exp_{{\xi_{\eps}}_i}^{-1} {\xi_{\eps}}_j}{\eps}\right)<\eps^4.
 	\end{equation}
 	
Finally, note that inequations (\ref{rho}) and 	(\ref{U4gamma}) imply the claim that  $\bar \xi_{\eps}\in D^K_{\eps,\rho}$.


 \end{proof}


 \section{Asymptotic expansion of the energy functional}\label{Asymptotic}
 In this section we prove Proposition \ref{Expansion}. We begin with the following.

 \begin{Lemma}
 	\label{Asympto}
 For $\bar \xi\in D^K_{\eps,\rho}$,	it holds
 	\begin{equation}
J_{\eps}(Y_{\eps,\bar \xi})=J_{\eps}\left(\sum_{i=1}^KW_{\eps,\xi_i}+\eps^2\sum_{i=1}^KV_{\eps, \xi_i}\right)
 	\end{equation} 	
$$=K \alpha +\eps^2  \frac{1}{2}   \beta \sum_{i=1}^K s_g(\xi_i) + \eps^4 \sum_{i=1}^K \Phi(\xi_i)-\frac{1}{2}\sum_{i\neq j}^{K}\gamma_{ij}U\left(\frac{\exp_{\xi_i}^{-1}(\exp_{\xi_j}(\eps
	z))}{\eps}\right)+o(\eps^4)$$

with 

$$\Phi(\xi):=\frac{1}{120(n+2)}   \left(  -c_8\Delta_{g} s_g(\xi)+c_6||Ric_{\xi}||^2-3c_1||R_{\xi}||^2  \right)
+c_7 s_g(\xi)^2+c_9 s_g(\xi)$$

 \begin{equation}
	\label{alphab}
	\alpha:= \frac{1}{2}\int_{\re^n}|\nabla U(z)|^2dz+\frac{1}{2}\int_{\re^n} U^2(z)dz-\frac{1}{p}\int_{\re^n} U^p(z)dz
\end{equation}

\begin{equation}
	\label{betab}
	\beta:= \mathbf{c} \int_{\re^n}U^2(z)dz-\frac{1}{n(n+2)}\int_{\re^n}|\nabla U(z)|^2|z|^2dz
\end{equation}

\begin{equation}
	\label{c1}
	c_1:=\frac{1}{6}\int_{\re^n} \left(\frac{U'(z)}{|z|}\right)^2 z_1^4 dz
\textnormal{, \ \  \  }
	c_2:=\int_{\re^n} U^2(z)z_1^2 dz,
\end{equation}

\begin{equation}
	\label{c3}
		c_3:=\frac{1}{54}\int_{\re^n }\psi(z)\frac{U'(z)}{|z|}z_1^4dz,
		\textnormal{\ \  \  \ \  }
c_4:=		- \frac{{\bf c}}{6}\int_{\re^n} U(z)\psi(z)z_1^2 dz 	
\end{equation}

 \begin{equation}
	\label{c5}
	c_5:=  \frac{{\bf c}}{6}\int_{\re^n}\left( \frac{1}{2} (U'(z))^2-\frac{1}{2-p} \frac{U(z)U'(z)}{|z|}\right)z_1^2dz 	
\end{equation}

   \begin{equation}
	\label{c6}
	c_6:=  8c_1-120(n+2)c_3
\textnormal{,\ \ \ \ \ \ }
	c_7:=  -c_3-c_4-c_5 - \frac{ c_2\mathbf{c}}{12}+\frac{c_1}{24(n+2)}
	\end{equation}

\begin{equation}
	\label{c8}
	c_8:= 18c_1+30 {\bf c} c_2(n+2) 
	\end{equation}
	
 \begin{equation}
	\label{c9}
	c_9:= \frac{1}{2}{\bf c}\int_{\re^n} U(z) \left( \frac{1}{2} U'(z)|z|-\frac{1}{2-p}U(z)\right)dz
\end{equation}
	
	\noindent where $\psi$ is the radial function discussed in Section \ref{Notation}, eq. (\ref{Psiog}).
 \end{Lemma}
 
 \begin{proof}
 	Let $\overline W_{\eps,\bar \xi}=\sum_{i=1}^K W_{\eps, \xi_i}$ and $\overline V_{\eps,\bar \xi}=\sum_{i=1}^K V_{\eps, \xi_i}$. By direct computation,  
 	$$
 		J_{\eps}(\overline W_{\eps,\bar \xi}+\eps^2\overline V_{\eps,\bar \xi})= 	J_{\eps}(\overline W_{\eps,\bar \xi})+\eps^2 	J_{\eps}'(\overline W_{\eps,\bar \xi})[\overline V_{\eps,\bar \xi}]+\frac{1}{2}\eps^4 J_{\eps}''(\overline W_{\eps,\bar \xi})[\overline V_{\eps,\bar \xi},\overline V_{\eps,\bar \xi}]+o(\eps^5)
$$

 \begin{equation}
	\label{overline}
	=J_{\eps}(\overline W_{\eps,\bar \xi})+\eps^2  \frac{1}{\eps^n} \int_{M}\left(-\eps^2  \Delta_g \overline W_{\eps,\bar \xi}+(1+\eps^2 \mathbf{c}s_g)\overline W_{\eps,\bar \xi}-|\overline W_{\eps,\bar \xi}|^{p-1}  \right) \overline V_{\eps,\bar \xi}\ d\mu_g
	\end{equation}
 	
 	$$
	+\eps^4  \frac{1}{2} \frac{1}{\eps^n}  \int_{M } \left(-\eps^2  \Delta_g \overline V_{\eps,\bar \xi}+(1+\eps^2 \mathbf{c}s_g)\overline V_{\eps,\bar \xi}-(p-1)|\overline W_{\eps,\bar \xi}|^{p-2}  \overline V_{\eps,\bar \xi} \right) \overline V_{\eps,\bar \xi}\ d\mu_g +o(\eps^5)
$$

Let $\mathcal{F}$ denote the last terms of (\ref{overline}), i.e., $J_{\eps}(\overline W_{\eps,\bar \xi}+\eps^2\overline V_{\eps,\bar \xi})=  J_{\eps}(\overline W_{\eps,\bar \xi})+\mathcal{F}$. 

 We first compute  $J_{\eps}\left(\overline W_{\eps,\bar \xi}\right)$:
$$J_{\eps}(\overline W_{\eps,\bar \xi})=J_{\eps}\left(\sum_{i=1}^K W_{\eps, \xi_i}\right)$$
$$=\frac{1}{\eps^n} \int_M \left[ \frac{1}{2}\eps^2|\nabla_g \sum_{i=1}^K W_{\eps, \xi_i}|^2+\frac{1}{2}(\eps^2 \mathbf{c}s_g+1)\left(\sum_{i=1}^K W_{\eps, \xi_i}\right)^2-\frac{1}{p}\left(\sum_{i=1}^K W_{\eps, \xi_i}\right)^p\right]d\mu_g$$

$$= \frac{1}{\eps^n}\sum_{i=1}^K\left[ \int_M  \frac{1}{2}\eps^2|\nabla_g  W_{\eps, \xi_i}|^2 d\mu_g+ \int_M  \frac{1}{2}(\eps^2 \mathbf{c}s_g+1)\left( W_{\eps, \xi_i}\right)^2d\mu_g- \int_M  \frac{1}{p}\left( W_{\eps, \xi_i}\right)^p d\mu_g\right]$$

$$+ \frac{1}{\eps^n}\sum_{i<j}^K\left[ \int_M \eps^2 \nabla_g  W_{\eps, \xi_i} \nabla_g  W_{\eps, \xi_j} d\mu_g+ \int_M  (\eps^2 \mathbf{c}s_g+1) W_{\eps, \xi_i} W_{\eps, \xi_j}d\mu_g- \int_M \left( W_{\eps, \xi_i}\right)^{p-1} W_{\eps, \xi_j} d\mu_g\right]$$

$$- \frac{1}{\eps^n}\left[\frac{1}{p} \int_M \left(\sum_{i=1}^K W_{\eps, \xi_i}\right)^p d\mu_g - \frac{1}{p} \sum_{i=1}^K\int_M \left( W_{\eps, \xi_i}\right)^p d\mu_g - \sum_{i<j}^K\int_M \left( W_{\eps, \xi_i}\right)^{p-1} W_{\eps, \xi_j} d\mu_g\right]$$
$$:=I_1+I_2+I_3$$

We now compute $I_1$:

$$I_1= \frac{1}{\eps^n}\sum_{i=1}^K\left[ \int_M  \frac{1}{2}\eps^2|\nabla_g  W_{\eps, \xi_i}|^2 d\mu_g+ \int_M  \frac{1}{2}\left( W_{\eps, \xi_i}\right)^2d\mu_g- \int_M  \frac{1}{p}\left( W_{\eps, \xi_i}\right)^p d\mu_g\right]$$

$$+ \frac{1}{\eps^n}\sum_{i=1}^K \int_M  \frac{1}{2}(\eps^2 \mathbf{c}s_g)\left( W_{\eps, \xi_i}\right)^2d\mu_g= A_1+A_2$$

It was computed in  \cite{DKM},  eq. (56), that 

\begin{equation}
	\label{A1}
A_1= K \alpha -\eps^2 c_1 \sum_{i=1}^K  s_g(\xi_i)+\eps^4 \sum_{i=1}^K \Theta(\xi_i)+o(\eps^4) 
\end{equation}

with

$$ 	\Theta(\xi)=\frac{1}{120(n+2)} c_1 \left( -18\Delta_{g} s_g(\xi)+8||Ric_{\xi}||^2+5 s_g(\xi)^2-3||R_{\xi}||^2\right).
$$

\noindent where $Ric_{\xi}$ is the Ricci curvature and $R_{\xi}$ the curvature tensor, meanwhile $\alpha$ and  $c_1$ are dimensional constants that depend only on $n$ and $m$, given by eqs. (\ref{alphab}) and (\ref{c1}).


We now estimate $ \frac{1}{2}   \frac{1}{\eps^n} \int_M  s_g W_{\eps, \xi_i}^2 d\mu_g$, in order  to estimate $A_2$. Consider  Taylor's expansions of $s_g$  around $\xi$,
$$
s_g(\exp_{\xi}(\eps z))=s_g(\xi)+  \eps \partial_j s_g(\xi) z_j+ \eps^2\frac{1}{2} \partial^2_{lj} s_g(\xi) z_l z_j+o(\eps^2)
$$

\noindent and recall the expansions of $g$ and $det(g)$   around $\xi$ (Lemma \ref{expansions}). We get

$$   \frac{1}{\eps^n}\frac{1}{2}  \int_M  s_g W_{\eps, \xi_i}^2d\mu_g= \frac{1}{2} \int_{B(0,r/\eps)} s_g(\exp_{\xi}(\eps z))(U(z)\chi_{r/\eps}(\eps z))^2 \sqrt{det(g_{\xi}(\eps z))}dz$$

 $$ =\frac{1}{2} s_g(\xi_i) \int_{\re^n} U^2 dz - \frac{\eps^2}{12} s_g(\xi_i) R_{kl}\int_{\re^n} U^2z_kz_l dz$$
 $$+\eps   \frac{1}{2} \partial_j s_g(\xi_i)   \int_{\re^n} U^2z_j dz+\eps^2   \frac{1}{4}   \partial^2_{lj} s_g(\xi_i)  \int_{\re^n} U^2 z_j z_ldz +o(\eps^2) $$

 
 Now, since $U^2$ is radial, for $i,j \in \{1, 2, \dots, n\}$, 
 $$ \int_{\re^n} U^2z_jz_l dz=\delta_{jl}\int_{\re^n} U^2z_jz_ldz = \delta_{jl}\int_{\re^n} U^2z_1^2 dz $$
 Also, for any $j \in \{1, 2, \dots, n\}$, $ \int_{\re^n} U^2z_j dz=  0 $, since $U$ is radial. This implies

\begin{equation}
	\label{A2}
	A_2=  \eps^2   \frac{1}{2} \mathbf{c} \left( \sum_{i=1}^K s_g(\xi_i) \right) \int_{\re^n} U^2 dz  -\frac{\eps^4}{4} \mathbf{c} \left( \frac{1}{3}\sum_{i=1}^K  s_g(\xi_i)^2+\sum_{i=1}^K \Delta s_g(\xi_i )\right)   \int_{\re^n} U^2z_1^2 dz     +o(\eps^4)  
\end{equation}

 Note that adding the terms of order $\eps^2$ from $A_1$ and $A_2$ gives 
 
 $$\frac{1}{2} \mathbf{c}  \int_{\re^n} U^2 dz  \sum_{i=1}^K s_g(\xi_i) -c_1  \sum_{i=1}^K  s_g(\xi_i) $$
 
 $$= \left( \frac{1}{2} \mathbf{c}  \int_{\re^n} U^2 dz  - \frac{1}{6}\int_{\re^n}\left(\frac{U'(|z|)}{|z|}\right)^2 z_1^4 dz  \right) \sum_{i=1}^K  s_g(\xi_i)$$
 $$=\frac{1}{2}  \left(\mathbf{c} \int_{\re^n} U^2(z)-\frac{1}{n(n+2)}\int|\nabla U(z)|^2|z|^2 dz\right) \sum_{i=1}^K  s_g(\xi_i)=\frac{1}{2}  \beta \sum_{i=1}^K  s_g(\xi_i)$$

 where we have used the equality (from \cite{RR}, page 14):
 $$\int_{\re^n} \left(\frac{U'(|z|)}{|z|}\right)^2 z_1^4 dz= \frac{3}{n(n+2)}\int_{\re^n} |\nabla U|^2 |z|^2 dz$$
 
 and $\beta$ is   the dimensional constant from equation (\ref{betab}). We conclude
 
 \begin{equation}
 	\label{I1}
 I_1=  K \alpha +\eps^2  \frac{1}{2}   \beta \sum_{i=1}^K s_g(\xi_i) +\eps^4 \left(\sum_{i=1}^K \Theta(\xi_i)  -\frac{ \mathbf{c} }{4}  c_2\left( \frac{1}{3}\sum_{i=1}^K  s_g(\xi_i)^2+\sum_{i=1}^K \Delta s_g(\xi_i )\right) \right) +o(\eps^4)
 \end{equation}
 
\noindent with $	c_2:=\int_{\re^n} U^2z_1^2 dz$. We now estimate $I_2$,
 
 $$I_2= \frac{1}{\eps^n}\sum_{i<j}^K\left[ \int_M \eps^2 \nabla_g  W_{\eps, \xi_i} \nabla_g  W_{\eps, \xi_j} d\mu_g+ \int_M   W_{\eps, \xi_i} W_{\eps, \xi_j}d\mu_g- \int_M \left( W_{\eps, \xi_i}\right)^{p-1} W_{\eps, \xi_j} d\mu_g\right]$$
 $$ +  \eps^2 \mathbf{c} \sum_{i<j}^K \frac{1}{\eps^n}\int_M s_g W_{\eps, \xi_i} W_{\eps, \xi_j}d\mu_g=:B_1+B_2$$

 It is shown in \cite{Dancer}, eq. (4.4) that for $i\neq j$,
 
 $$B_1=\frac{1}{\eps^n}\sum_{i<j}^K\left[ \int_M \eps^2 \nabla_g  W_{\eps, \xi_i} \nabla_g  W_{\eps, \xi_j} d\mu_g+ \int_M   W_{\eps, \xi_i} W_{\eps, \xi_j}d\mu_g- \int_M \left( W_{\eps, \xi_i}\right)^{p-1} W_{\eps, \xi_j} d\mu_g\right]$$
 $$=o(\eps^2 e^{-(1-\delta) \frac{d_g(\xi_i,\xi_j)}{\eps}})$$
 
 which is of order $o(\eps^4)$, for $\xi_i, \xi_j$ components of $\bar \xi \in D^K_{\eps,\rho}$, by (\ref{delta2}).
 
 On the other hand, since $s_g$ is bounded, for some $C_s$ upper bound for $s_g$:
 
 $$  \eps^2 \mathbf{c} \sum_{i<j}^K \frac{1}{\eps^n}\int_M s_g W_{\eps, \xi_i} W_{\eps, \xi_j}d\mu_g< \eps^2 C_s \mathbf{c} \sum_{i<j}^K \frac{1}{\eps^n}\int_M W_{\eps, \xi_i} W_{\eps, \xi_j}d\mu_g$$
 
 Let $\exp_{\xi_i}^{-1}(x)=\eps z$. Then, for some $C>0$, $\delta>0$ small, $i\neq j$,
 
 $$\frac{1}{\eps^n}\int_M W_{\eps, \xi_i} W_{\eps, \xi_j}d\mu_g$$
 $$=\int_{B(0,r/\eps)}U(z)\chi_r(z) U\left(\frac{\exp_{\xi_i}^{-1}(\exp_{\xi_j}(\eps
 	z))}{\eps}\right)\chi_r(\exp_{\xi_i}^{-1}(\exp_{\xi_j}(\eps z)))|g_{\xi_i}(\eps z)|dz$$
 $$\leq C \int_{B(0,r/\eps)}U(z) U\left(\frac{\exp_{\xi_i}^{-1}(\exp_{\xi_j}(\eps z))}{\eps}\right)dz= o(e^{-(1-\delta) \frac{d_g(\xi_i,\xi_j)}{\eps}})$$
 
 \noindent where in the last equality we have used an estimation from Lemma \ref{5.1}.   Note that  $o\left(e^{-(1-\delta) \frac{d_g(\xi_i,\xi_j)}{\eps}}\right)=o(\eps^4)$, for $\xi_i, \xi_j$ components of $\bar \xi \in D^K_{\eps,\rho}$, by (\ref{delta2}). This implies $I_2=o(\eps^4)$. We now claim 
 
 \begin{equation}
 	\label{I3}
 	I_3=-\frac{1}{2}\sum_{i\neq j}^{K}\gamma_{ij}U\left(\frac{\exp_{\xi_i}^{-1}(\exp_{\xi_j}(\eps
 		z))}{\eps}\right)+o(\eps^4)
 \end{equation}
 
 \noindent where $\gamma_{ij}$ is defined in (\ref{gamma}).

 It is shown in \cite{Dancer}, eq. (4.7) that 
 
 $$I_3=-\frac{1}{\eps^n}\left[\frac{1}{p} \int_M \left(\sum_{i=1}^K W_{\eps, \xi_i}\right)^p d\mu_g - \frac{1}{p} \sum_{i=1}^K\int_M \left( W_{\eps, \xi_i}\right)^p d\mu_g - \sum_{i<j}^K\int_M \left( W_{\eps, \xi_i}\right)^{p-1} W_{\eps, \xi_j} d\mu_g\right]$$
 $$= \frac{1}{\eps^n} \int_M\sum_{j=1}^{K-1}  \left( \sum_{i=j+1}^K W_{\eps, \xi_i}\right)^{p-1} W_{\eps, \xi_j}  d\mu_g+o(e^{-(1-\delta) \frac{d_g(\xi_i,\xi_j)}{\eps}})$$
 
 and in eq. (4.8), (4.10) and (4.11) of the same article, that 
 
 $$ \frac{1}{\eps^n} \int_M  \left( \sum_{i=j+1}^K W_{\eps, \xi_i}\right)^{p-1} W_{\eps, \xi_j}  d\mu_g=\sum_{i\neq j}^{K}\gamma_{ij}U\left(\frac{\exp_{\xi_i}^{-1}(\exp_{\xi_j}(\eps 	z))}{\eps}\right)+o\left(e^{-p\frac{d_g(\xi_i,\xi_j)}{2\eps}}\right)$$

 Taking into account that $o(e^{-(1-\delta) \frac{d_g(\xi_i,\xi_j)}{\eps}})=o(\eps^4)$ and $o(e^{-p\frac{d_g(\xi_i,\xi_j)}{2\eps}})=o(\eps^4)$, by (\ref{delta2}), this implies the claim of eq. (\ref{I3}).

 We now compute the terms corresponding to $\mathcal{F}$, from eq. (\ref{overline}).
 
$$\mathcal{F}:= \eps^2  \frac{1}{\eps^n} \int_{M}\left(-\eps^2  \Delta_g \sum_{i=1}^K W_{\eps,\xi_i}+(1+\eps^2 \mathbf{c}s_g)\sum_{i=1}^K W_{\eps,\xi_i}-|\sum_{i=1}^K W_{\eps,\xi_i}|^{p-1}  \right) \sum_{j=1}^K V_{\eps,\xi_j}\ d\mu_g
$$
 
 $$
 +\eps^4  \frac{1}{2} \frac{1}{\eps^n}  \int_{M } \left(-\eps^2  \Delta_g \sum_{i=1}^K V_{\eps,\xi_i}+(1+\eps^2 \mathbf{c}s_g)\sum_{i=1}^K V_{\eps,\xi_i}-(p-1)|\sum_{i=1}^K W_{\eps,\xi_i}|^{p-2}  \sum_{l=1}^K V_{\eps,\xi_l} \right) \sum_{j=1}^K V_{\eps,\xi_j}\ d\mu_g +o(\eps^5)
 $$

 $$
=\sum_{i=1}^K\left[ \eps^2  \frac{1}{\eps^n} \int_{M}   \left(-\eps^2 \Delta_g   W_{\eps,\xi_i}+(1+\eps^2 \mathbf{c}s_g)  W_{\eps,\xi_i}-|  W_{\eps,\xi_i}|^{p-1}  \right)  V_{\eps,\xi_i}\ d\mu_g\right]
$$

$$
+\sum_{i=1}^K\left[ \eps^4  \frac{1}{2} \frac{1}{\eps^n}  \int_{M } \left(-\eps^2 \Delta_g   V_{\eps,\xi_i}+(1+\eps^2 \mathbf{c}s_g)  V_{\eps,\xi_i}-(p-1)|  W_{\eps,\xi_i}|^{p-2}    V_{\eps,\xi_i} \right)   V_{\eps,\xi_i}\ d\mu_g\right] 
$$

$$
+ \eps^2 \frac{1}{\eps^n}\sum_{i\neq j}^K \left[\int_{M}\eps^2   \nabla W_{\eps,\xi_i},\nabla V_{\eps,\xi_j}  d\mu_g +(1+\eps^2 \mathbf{c}s_g) W_{\eps,\xi_i}V_{\eps,\xi_j} d\mu_g-|  W_{\eps,\xi_i}|^{p-1} V_{\eps,\xi_j} d\mu_g \right]  
$$

$$
+\eps^4  \frac{1}{2} \frac{1}{\eps^n} \sum_{i\neq j}^K  \int_{M } \eps^2  \nabla V_{\eps,\xi_i},\nabla V_{\eps,\xi_j} d\mu_g +\int_M(1+\eps^2 \mathbf{c}s_g)  V_{\eps,\xi_i} V_{\eps,\xi_j} d\mu_g -\int_M(p-1)|  W_{\eps,\xi_i}|^{p-2}  V_{\eps,\xi_i} V_{\eps,\xi_j}  d\mu_g  
$$

$$+ \eps^2  \frac{1}{\eps^n} \left[\sum_{j=1}^K  \int_M -|\sum_{i=1}^K W_{\eps,\xi_i}|^{p-1}    V_{\eps,\xi_j} d\mu_g+  \sum_{i=1}^{K} \int_{M}    |  W_{\eps,\xi_i}|^{p-1}     V_{\eps,\xi_i}\ d\mu_g +  \sum_{i\neq j}^K \int_M |  W_{\eps,\xi_i}|^{p-1} V_{\eps,\xi_j}d\mu_g \right] $$

$$ +\eps^4\frac{(p-1) }{2\eps^n}   \left[   -  \sum_{j,l=1}^K\int_{M }  |\sum_{i\neq j}^K  W_{\eps,\xi_i}|^{p-2}  V_{\eps,\xi_j}  V_{\eps,\xi_l} d\mu_g   +     \sum_{i=1}^K   \int_{M }|  W_{\eps,\xi_i}|^{p-2}    V_{\eps,\xi_i}^2     d\mu_g
+  \sum_{i\neq j}^K   \int_M |  W_{\eps,\xi_i}|^{p-2}  V_{\eps,\xi_i} V_{\eps,\xi_j}  d\mu_g     
\right]$$
$$  +o(\eps^5)$$

$$=:K_1+K_2+K_3+K_4+K_5+K_6+o(\eps^5)$$

 We first consider $K_1+K_2$. Note that

 $$K_1+K_2  =\sum_{i=1}^K\left[ \eps^2  \frac{1}{\eps^n} \int_{M}   \left(-\Delta_g   W_{\eps,\xi_i}+ (1+{\bf c}s_g \eps^2)W_{\eps,\xi_i}-|  W_{\eps,\xi_i}|^{p-1}  \right)  V_{\eps,\xi_i}\ d\mu_g\right]
 $$
 
 $$ +\sum_{i=1}^K\left[ \eps^4  \frac{1}{2} \frac{1}{\eps^n}  \int_{M } \left(-\Delta_g   V_{\eps,\xi_i}+  V_{\eps,\xi_i}-(p-1)|  W_{\eps,\xi_i}|^{p-2}    V_{\eps,\xi_i} \right)   V_{\eps,\xi_i}\ d\mu_g\right]  $$
 $$+ \sum_{i=1}^K\ \eps^6  \frac{1}{2} \frac{1}{\eps^n}  \int_{M }   \mathbf{c}s_g   V_{\eps,\xi_i}^2   d\mu_g   +o(\eps^4)$$

 \noindent  and since $V_{\eps,\xi_i}$ is bounded and exponentially decreasing,
 
 $$\eps^6  \frac{1}{2} \frac{1}{\eps^n}  \int_{M }   \mathbf{c}s_g   V_{\eps,\xi_i}^2   d\mu_g =o(\eps^5)$$
 
 Hence 
 
\begin{equation}
	\label{K1K2}
	K_1+K_2 =\sum_{i=1}^K \mathcal{A}(\xi_i) +o(\eps^4) 
\end{equation}
 
 \noindent where 
 
 $$\mathcal{A}(\xi):= \eps^2  \frac{1}{\eps^n} \int_{M}   \left(-\Delta_g   W_{\eps,\xi}+  (1+{\bf c}s_g \eps^2) W_{\eps,\xi}-|  W_{\eps,\xi}|^{p-1}  \right)  V_{\eps,\xi}\ d\mu_g $$
 \begin{equation}
 	\label{mathcalA}
 + \eps^4  \frac{1}{2} \frac{1}{\eps^n}  \int_{M } \left(-\Delta_g   V_{\eps,\xi}+  V_{\eps,\xi}-(p-1)|  W_{\eps,\xi}|^{p-2}    V_{\eps,\xi} \right)   V_{\eps,\xi}  d\mu_g  +o(\eps^5)
  \end{equation}

 Note that by construction of $V$, eq.  (\ref{V}), and using eq. (\ref{Kwongeps}) for $U$, the term  
 
 $$\frac{1}{\eps^n} \int_{M}   \left(-\Delta_g   W_{\eps,\xi}+ (1+{\bf c}s_g \eps^2) W_{\eps,\xi}-|  W_{\eps,\xi}|^{p-1}  \right)  V_{\eps,\xi}\ d\mu_g$$
  is  of order $\eps^2$, more specifically
 \begin{equation}
	\label{mathcalAB}
	  \eps^2  \frac{1}{\eps^n} \int_{M}   \left(-\Delta_g   W_{\eps,\xi}+ (1+{\bf c}s_g \eps^2) W_{\eps,\xi}-|  W_{\eps,\xi}|^{p-1}  \right)  V_{\eps,\xi}\ d\mu_g =-\eps^4 \int_{\re^n} L_0 VVdz +o(\eps^4)
	\end{equation} 

Meanwhile,  using the fact that  $\Delta_g V-\Delta_{\re^n}V=o(\eps^2)$ and rewriting the second term of (\ref{mathcalA}) on $\re^n$:
 
 \begin{equation}
	\label{mathcalA2}
	    \frac{\eps^4}{2\eps^n}  \int_{M } \left(-\Delta_g   V_{\eps,\xi}+  V_{\eps,\xi}-(p-1)|  W_{\eps,\xi}|^{p-2}    V_{\eps,\xi} \right)   V_{\eps,\xi}  d\mu_g =   \frac{\eps^4}{2} \int_{\re^n} L_0 VVdz+o(\eps^4)
\end{equation}

Hence, by adding (\ref{mathcalAB}) and (\ref{mathcalA2}),  (\ref{mathcalA}) turns into

\begin{equation}
	\label{mathcalA3}
	\mathcal{A}(\xi)=-\eps^4  \frac{1}{2}\int_{\re^n} L_0 VVdz +o(\eps^4)
 \end{equation}

Recall also that $V=V_1+V_2$, with $V_1$ and $V_2$ given by eqs. (\ref{Vestimation}) and (\ref{Vestimation2}). Hence

 \begin{equation}
 	\label{LVV}
 	\int_{\re^n} L_0 VVdz= \int_{\re^n} L_0 V_1V_1dz+\int_{\re^n} L_0 V_1V_2dz+\int_{\re^n} L_0 V_2V_1dz+\int_{\re^n} L_0 V_2V_2dz
 \end{equation}

The first term, $\int_{\re^n} L_0 V_1V_1dz$, can be computed as follows, using eq. (\ref{Lov2}) for $L_0V_1$ and eq. (\ref{Vestimation}) for $V_1$,

$$-\eps^4  \frac{1}{2}\int_{\re^n} L_0 V_1V_1dz= 	-\eps^4  \frac{1}{2}\frac{1}{9}\left(  R_{kl} R_{ij}\int_{\re^n} \psi\frac{U'}{|z|} z_k z_l z_i z_j dz \right)$$
\begin{equation}
	\label{Lv1v1}
=-\eps^4 c_3 \left( 
||Ric_{\xi}||^2+ s_g(\xi)^2\right)
\end{equation}

\noindent since $ \psi\frac{U'}{|z|}$ is radial; with $c_3$ a  dimensional constant, given by

\begin{equation}
	\label{c34}
	c_3:=\frac{1}{54}\int_{\re^n }\psi(|z|)\frac{U'(|z|)}{|z|}z_1^4dz,
\end{equation}
 \noindent where $\psi=\psi(|z|)$ is the radial  function  discussed in Section \ref{Notation}, eqs. (\ref{Psiog}) and (\ref{Vestimation}). Note that for (\ref{Lv1v1}), we have used  the fact that for any radial function $\varphi(|z|)$ in $\re^n$,
 
 $$\int_{\mathbb{R}^n} \varphi(|z|) z_i z_j z_k z_l  dz = (\delta_{ij}\delta_{kl} + \delta_{ik}\delta_{jl} + \delta_{il}\delta_{jk}) \int_{\mathbb{R}^n} \varphi(|z|) z_1^2 z_2^2$$
$$  =\frac{1}{3} (\delta_{ij}\delta_{kl} + \delta_{ik}\delta_{jl} + \delta_{il}\delta_{jk}) \int_{\mathbb{R}^n} \varphi(|z|) z_1^4$$
 
\noindent which follows from the fact that 

$$  \int_{\mathbb{S}^{n-1}} z_i z_j z_k z_l =  (\delta_{ij}\delta_{kl} + \delta_{ik}\delta_{jl} + \delta_{il}\delta_{jk})  \int_{\mathbb{S}^{n-1}} z_1^2 z_2^2 =\frac{1}{3} (\delta_{ij}\delta_{kl} + \delta_{ik}\delta_{jl} + \delta_{il}\delta_{jk})  \int_{\mathbb{S}^{n-1}} z_1^4$$

In a similar way, for any radial function $\varphi(|z|)$ on $\re^n$,

\begin{equation}
	\label{radial2}
	\int_{\re^n} \varphi(|z|) z_kz_l dz = \delta_{kl}\int_{\re^n} \varphi(|z|) z_1^2 dz 
\end{equation}
  
 We now compute the rest of the terms in (\ref{LVV}).  For the second term in (\ref{LVV}), by direct computation, using eq. (\ref{Lov2}) for $L_0V_1$ and eq. (\ref{Vestimation2}) for $V_2$, 
 $$\int_{\re^n} L_0 V_1V_2dz=\frac{1}{3}R_{kl}{\bf c}s_g\int_{\re^n} \frac{U'(z)}{|z|}  \left( \frac{1}{2} U'(z)|z|-\frac{1}{2-p}U(z)\right)z_kz_ldz$$
  $$=\frac{{\bf c}}{3}s_g^2\int_{\re^n}\left( \frac{1}{2} (U'(z))^2-\frac{1}{2-p} \frac{U(z)U'(z)}{|z|}\right)z_1^2dz $$
 
 \noindent by (\ref{radial2}), since $\left( \frac{1}{2} (U'(z))^2-\frac{1}{2-p} \frac{U(z)U'(z)}{|z|}\right)$ is a radial function. On the other hand,  by eq. (\ref{LovU2}) for $L_0V_2$ and eq. (\ref{Vestimation}) for $V_1$, 
  $$\int_{\re^n} L_0 V_2V_1dz=-\frac{1}{3}{\bf c}s_g R_{kl}\int_{\re^n} U(z)\psi(|z|) z_kz_l  dz  $$
 
 $$=-\frac{1}{3}{\bf c}s_g R_{kl} \delta_{kl}\int_{\re^n} U(z)\psi(|z|)z_1^2  dz=-\frac{1}{3}{\bf c}s_g^2\int_{\re^n} U(z)\psi(|z|)z_1^2 dz =2 c_4 s_g^2 $$
 
with $c_4:=-\frac{1}{6}{\bf c}\int_{\re^n} U(z)\psi(|z|)z_1^2 dz $, since $ U(z)\psi(|z|)$   a radial function. And
by eq. (\ref{LovU2}) for $L_0V_2$ and eq. (\ref{Vestimation2}) for $V_2$:

 $$\int_{\re^n} L_0 V_2V_2dz=-{\bf c}s_g \int_{\re^n} U(z) \left( \frac{1}{2} U'(z)|z|-\frac{1}{2-p}U(z)\right)dz$$

  We then get, from these equations and (\ref{K1K2}),  together with (\ref{mathcalA3}) and (\ref{LVV}),

$$K_1+K_2 =	\mathcal{A}(\xi)+o(\eps^4)=-\eps^4  \frac{1}{2}\int_{\re^n} L_0 VVdz +o(\eps^4)$$
\begin{equation}
	\label{mathcalA4}
=-\eps^4 \left(c_3 
||Ric_{\xi}||^2+c_3 s_g(\xi)^2+ (c_5+c_4) s_g(\xi)^2- c_9s_g(\xi)\right)+o(\eps^4)
\end{equation}

\noindent with dimensional constants $c_5$, $c_9$ given by,

 \begin{equation}
 	\label{c5b}
 c_5:=\frac{{\bf c}}{6}\int_{\re^n}\left( \frac{1}{2} (U'(z))^2-\frac{1}{2-p} \frac{U(z)U'(z)}{|z|}\right)z_1^2dz 	
  \end{equation}
  
   \begin{equation}
   	\label{c9b}
  	c_9:= \frac{1}{2}{\bf c}\int_{\re^n} U(z) \left( \frac{1}{2} U'(z)|z|-\frac{1}{2-p}U(z)\right)dz
  \end{equation}

On the other hand, $K_3$, $K_4$, $K_5$ and $K_6$ are $o(\eps^4)$ by Lemmas  \ref{K3}, \ref{K4}, \ref{K5} and \ref{K6}, respectively. 
 We conclude,
$$ J_{\eps}\left(\sum_{i=1}^KW_{\eps,\xi_i}+\eps^2\sum_{i=1}^KV_{\eps, \xi_i}\right)=K \alpha +\eps^2  \frac{1}{2}   \beta \sum_{i=1}^K s_g(\xi_i) + \eps^4 \sum_{i=1}^K \Phi(\xi_i)-\frac{1}{2}\sum_{i\neq j}^{K}\gamma_{ij}U\left(\frac{\exp_{\xi_i}^{-1}(\exp_{\xi_j}(\eps
	z))}{\eps}\right)+o(\eps^4)$$

with 

$$\Phi(\xi):=\frac{1}{120(n+2)}   \left(  -c_8\Delta_{g} s_g(\xi)+c_6||Ric_{\xi}||^2-3c_1||R_{\xi}||^2  \right)
+c_7 s_g(\xi)^2+c_9 s_g(\xi)$$

\noindent  where, $c_6= 8c_1-120(n+2)c_3$, $c_7= -c_3-c_4- c_5 - \frac{ c_2\mathbf{c}}{12}+\frac{c_1}{24(n+2)}$        and $c_8= 18c_1+30 {\bf c} c_2(n+2) $.

 \end{proof}

 We now are ready to prove Proposition \ref{Expansion}.
 
 \begin{proof}\textnormal{(of Proposition  \ref{Expansion})}
 	
 	 Recall that $Y_{\eps, \bar\xi}=\sum_{i=1}^K (W_{\eps,\xi_i}+\eps^2 V_{\eps,\xi_i})$. We will prove that
 	 
 	 \begin{equation}
 	 	J_{\eps}\left(Y_{\eps, \bar\xi}+\phi_{\eps,\bar \xi} \right)-J_{\eps}\left(Y_{\eps, \bar\xi} \right)=o(\eps^4)
 	 	\end{equation}
\noindent and the Proposition will follow from Lemma \ref{Asympto}. We write $F(u)=\frac{1}{p}(u^+)^{p}$ and  $f(u)=(u^+)^{p-1}$.	 
First, since $\phi_{\eps,\bar \xi}\in K^{\perp}_{\eps, \bar\xi}$,

$$0=\langle \phi_{\eps,\bar \xi}, S_{\eps}(Y_{\eps, \bar\xi}+\phi_{\eps,\bar \xi})\rangle_{\eps}=\langle \phi_{\eps,\bar \xi}, (Y_{\eps, \bar\xi}+\phi_{\eps,\bar \xi})-i^*_{\eps} (f(Y_{\eps, \bar\xi}+\phi_{\eps,\bar \xi}))\rangle_{\eps} $$

$$=	 \frac{1}{2}||\phi_{\eps,\bar \xi} 	||^2_{\eps}+\frac{1}{\eps^n}\int_M [\eps^2\nabla_g Y_{\eps, \bar\xi}\nabla_g \phi_{\eps,\bar \xi}+(1+{\bf c}s_g \eps^2)Y_{\eps, \bar\xi}\phi_{\eps,\bar \xi}-f(Y_{\eps, \bar\xi}+\phi_{\eps,\bar \xi})\phi_{\eps,\bar \xi}]d\mu_g$$ 

Hence, by direct computation and the last equation,
	 $$J_{\eps}\left(Y_{\eps, \bar\xi}+\phi_{\eps,\bar \xi} \right)-J_{\eps}\left(Y_{\eps, \bar\xi} \right)=
	 \frac{1}{2}||\phi_{\eps,\bar \xi} 	||^2_{\eps}-\frac{1}{\eps^n}\int_M [F(Y_{\eps, \bar\xi}+\phi_{\eps,\bar \xi})-F(Y_{\eps, \bar\xi})]d\mu_g$$
	 
	 $$=\frac{1}{2}||\phi_{\eps,\bar \xi} 	||^2_{\eps}+\frac{1}{\eps^n}\int_M [\eps^2\nabla_g Y_{\eps, \bar\xi}\nabla_g \phi_{\eps,\bar \xi}+(1+{\bf c}s_g \eps^2)Y_{\eps, \bar\xi}\phi_{\eps,\bar \xi}-f(Y_{\eps, \bar\xi})\phi_{\eps,\bar \xi}]d\mu_g$$ 
	 $$-\frac{1}{\eps^n}\int_M [F(Y_{\eps, \bar\xi}+\phi_{\eps,\bar \xi})-F(Y_{\eps, \bar\xi})-f(Y_{\eps, \bar\xi})\phi_{\eps,\bar \xi}]d\mu_g$$
	 
	 $$=-\frac{1}{2}||\phi_{\eps,\bar \xi} 	||^2_{\eps}+\frac{1}{\eps^n}\int_M [f(Y_{\eps, \bar\xi}+\phi_{\eps,\bar \xi})-f(Y_{\eps, \bar\xi})]\phi_{\eps,\bar \xi}d\mu_g
	  	 	-\frac{1}{\eps^n}\int_M [F(Y_{\eps, \bar\xi}+\phi_{\eps,\bar \xi})-F(Y_{\eps, \bar\xi})-f(Y_{\eps, \bar\xi})\phi_{\eps,\bar \xi}]d\mu_g$$
 	 
 	$$=-\frac{1}{2}||\phi_{\eps,\bar \xi} 	||^2_{\eps}+\frac{1}{\eps^n}\int_M f'(Y_{\eps, \bar\xi}+t_1\phi_{\eps,\bar \xi})\phi^2_{\eps,\bar \xi}d\mu_g-\frac{1}{2\eps^n}\int_M f'(Y_{\eps, \bar\xi}+t_2\phi_{\eps,\bar \xi})\phi^2_{\eps,\bar \xi}d\mu_g$$
 	
 Since,	for some $t_1,t_2\in [0,1]$, by the mean value Theorem,
 	$$\frac{1}{\eps^n}\int_M [f(Y_{\eps, \bar\xi}+\phi_{\eps,\bar \xi})-f(Y_{\eps, \bar\xi})]\phi_{\eps,\bar \xi}d\mu_g
= 	\frac{1}{\eps^n}\int_M f'(Y_{\eps, \bar\xi}+t_1\phi_{\eps,\bar \xi})\phi^2_{\eps,\bar \xi}d\mu_g$$
 	 
 	and 

	 $$-\frac{1}{\eps^n}\int_M [F(Y_{\eps, \bar\xi}+\phi_{\eps,\bar \xi})-F(Y_{\eps, \bar\xi})-f(Y_{\eps, \bar\xi})\phi_{\eps,\bar \xi}]d\mu_g=\frac{1}{2\eps^n}\int_M f'(Y_{\eps, \bar\xi}+t_2\phi_{\eps,\bar \xi})\phi^2_{\eps,\bar \xi}d\mu_g$$

 	  Moreover, for any $t\in[0,1]$,
 	$$\frac{1}{\eps^n}\int_M f'(Y_{\eps, \bar\xi}+t\phi_{\eps,\bar \xi})\phi^2_{\eps,\bar \xi}d\mu_g\leq c \frac{1}{\eps^n}\int_M Y_{\eps, \bar\xi}^{p-2}\phi_{\eps,\bar \xi}^2d\mu_g+ c \frac{1}{\eps^n}\int_M \phi_{\eps,\bar \xi}^p d\mu_g$$
 	$$\leq C \frac{1}{\eps^n}\int_M \phi_{\eps,\bar \xi}^2d\mu_g+ c \frac{1}{\eps^n}\int_M \phi_{\eps,\bar \xi}^p d\mu_g\leq c \left(||\phi_{\eps,\bar \xi} 	||^2_{\eps}+||\phi_{\eps,\bar \xi} 	||^p_{\eps}\right)=o(\eps^4)$$
 	
\noindent where,  we have used, from eq. (\ref{barfi}), that  $||\phi_{\eps,\bar \xi} 	||_{\eps}=o(\eps^2)$, since $\bar \xi \in D^K_{\eps,\rho}$.



 \end{proof}
 


\section{The Finite dimensional reduction} \label{FiniteDimensional}
In this section we sketch a proof for the finite dimensional reduction, Prop. \ref{Existence}.  We follow \cite{Dancer, Deng, RR}, where more detailed proofs in similar situations can be found.

Recall that one of our  objectives is to solve eq. $(\ref{perp})$, which we may rewrite  as

$$ 0= \Pi_{\eps,\xi}^{\perp}\left\{ S_{\eps}(Y_{\eps,\bar \xi}+\phi)\right\}=  \Pi_{\eps,\xi}^{\perp}\left\{ S_{\eps}(Y_{\eps,\bar \xi} )+ S'_{\eps}(Y_{\eps,\bar \xi}) \phi+ \bar N_{\eps,\bar \xi}(\phi)\right\}	=-R_{\eps,\bar \xi}  + L_{\eps,\bar \xi} - N_{\eps,\bar \xi}(\phi)$$

with 
\begin{equation}
	R_{\eps,\bar \xi}:=- \Pi_{\eps,\xi}^{\perp}\left\{ S_{\eps}(Y_{\eps,\bar \xi} )\right\}= \Pi_{\eps,\xi}^{\perp}\left\{i^*_{\eps}(f(Y_{\eps,\bar \xi}))-Y_{\eps,\bar \xi} \right\}
\end{equation}

\begin{equation}
	L_{\eps,\bar \xi}(\phi):= \Pi_{\eps,\xi}^{\perp}\left\{ S'_{\eps}(Y_{\eps,\bar \xi} )\phi\right\}= \Pi_{\eps,\xi}^{\perp}\left\{\phi -i^*_{\eps}(f'(Y_{\eps,\bar \xi})\phi)\right\}
\end{equation}

\begin{equation}
	N_{\eps,\bar \xi}(\phi):= \Pi_{\eps,\xi}^{\perp}\left\{ \bar N_{\eps,\bar \xi}(\phi)\right\}= \Pi_{\eps,\xi}^{\perp}\left\{i^*_{\eps}\left(f(Y_{\eps,\bar \xi}+\phi)-f(Y_{\eps,\bar \xi}) -f'(Y_{\eps,\bar \xi})\phi\right) \right\}
\end{equation}

In this way, we can rewrite eq. (\ref{perp}) as $L_{\eps,\bar \xi}(\phi) =	  N_{\eps,\bar \xi}(\phi)+R_{\eps,\bar \xi}  $.

Moreover, if $L_{\eps,\bar \xi}$ is invertible, the problem in eq. (\ref{perp}) turns into a fixed point problem for the operator,
$$T_{\eps,\bar \xi}(\phi) = L^{-1}_{\eps,\bar \xi}  \left( N_{\eps,\bar \xi}(\phi)+R_{\eps,\bar \xi}  \right)$$

Following this idea, we begin by proving that $L_{\eps,\bar \xi}$ is invertible, for suitable $\bar \xi$ and $\eps$.

\begin{Lemma}
	\label{Finite1}
	There is some $\eps_0$, $c>0$, such that for any $\eps\in (0,\eps_0)$ and $\bar \xi \in M^K$, $\bar \xi=(\xi_1, \xi_2,\dots, \xi_K)$, such that if 
	
	\begin{equation}
		\label{Ueps4}
		\sum_{\substack{i,j=1\\ i\neq j}}^K U\left(\frac{exp_{\xi_i}^{-1}\xi_j}{\eps}\right)<\eps^4,
	\end{equation}
	
	\noindent	then, for any $\phi \in K^{\perp}_{\eps,\bar \xi}$,  $||L_{\eps,\bar \xi}(\phi)||_{\eps}\geq c ||\phi||_{\eps}.$

\end{Lemma}

\begin{proof}
	We follow \cite[Proposition 3.1]{Micheletti}, \cite[Lemma 6.1]{RR}. We proceed by contradiction. Suppose there is a sequence $\eps_j\rightarrow 0$, $\{\eps_j\}_{j\in \N}$	together with $\{\bar \xi_j\}_{j\in\N}\subset M^K$, such that the hypothesis of the Lemma, eq. (\ref{Ueps4}) is satisfied
	and at the same time there is some $\{\phi_j\}_{j\in \N}\subset  K^{\perp}_{\eps,\bar \xi}$, with 
	\begin{equation}
		\label{phi1}
		||\phi_j||_{\eps_j}=1,  \text{ such that }||\psi_j||_{\eps_j}\rightarrow0,
	\end{equation}
	
	\noindent where  $\psi_j:=L_{\eps,\bar \xi}(\phi_j)$. Hence
	
	\begin{equation}
		\label{psizeta}
		\phi_j-i^*_{\eps_j}(f'(Y_{\eps_j,\bar \xi_j})\phi_j)=\psi_j+\zeta_j
	\end{equation}
	
	\noindent with $Y_{\eps,\bar \xi}= \sum_{i=1}^{K}  \left( U_{\eps, \xi_i}+\eps^2 V_{\eps,\xi_i}\right)$ and  $\zeta_j:=\Pi_{\eps,\xi} \left\{\phi_j-i^*_{\eps_j}(f'(Y_{\eps_j,\bar \xi_j})\phi_j)\right\}$. Let $w_j=\phi_j-(\psi_j+\zeta_j)$. We will prove that the existence of such sequences  lead to the following contradiction
	
	\begin{equation}
		\label{convto1}
		\frac{1}{\eps_j^n}\int_M  f'(Y_{\eps_j,\bar \xi_j})w_j^2 d\mu_g\rightarrow 1
	\end{equation}
	and 
	\begin{equation}
		\label{convto0}
		\frac{1}{\eps_j^n}\int_M  f'(Y_{\eps_j,\bar \xi_j})w_j^2 d\mu_g\rightarrow 0
	\end{equation}
	
	We first prove eq. (\ref{convto1}).
	Being $\zeta_j \in K_{\eps_j,\bar \xi_j}$ we write 	$\zeta_j =\sum_{i=1}^K \sum_{m=1}^n a_{j}^{mi}Z_{\eps_j,\bar \xi_j}^m$. For $m\in \{1,2,\dots,n\}$ and $l \in \{1,2,\dots, K\}$, we multiply $\psi_j+\zeta_j$ by $Z_{\eps_j,\bar \xi_j}^h$, $h \in \{1,2,\dots,n\}$ (using eq. (\ref{psizeta}))
	
	\begin{equation}
		\label{69}
		\sum_{i=1}^K \sum_{m=1}^n a_{j}^{mi}\langle Z_{\eps_j,\bar \xi_j}^m, Z_{\eps_j,\bar \xi_j}^h \rangle_{\eps_j}= \langle \phi_j, Z_{\eps_j,\bar \xi_j}^h \rangle_{\eps_j}- \langle i^*_{\eps_j}(f'(Y_{\eps_j,\bar \xi_j})\phi_j), Z_{\eps_j,\bar \xi_j}^h \rangle_{\eps_j}
	\end{equation}	
	on the other hand, by (\ref{ZZ})
	
	\begin{equation}
		\label{70}
		\sum_{i=1}^K \sum_{m=1}^n a_{j}^{mi}\langle Z_{\eps_j,\bar \xi_j}^m, Z_{\eps_j,\bar \xi_j}^h \rangle_{\eps_j}= C a_j^{hl}+o(1)
	\end{equation}		
	
	Hence, by (\ref{69}) and (\ref{70})
	\begin{equation}
		C a_j^{hl}+o(1)=  \frac{1}{\eps_j^n}  \int_M\left[ \eps_j^2 \nabla_g  Z_{\eps_j,\bar \xi_j}^h \nabla_g  \phi_j +(\eps_j^2 \mathbf{c}s_g+1)  Z_{\eps_j,\bar \xi_j}^h \phi_j -  f'(Y_{\eps_j,\bar \xi_j}) Z_{\eps_j,\bar \xi_j}^h \phi_j  \right]  d\mu_g 
	\end{equation}

	Let $	\tilde	{\phi_{l}}_{j}:\re^n\rightarrow \re$ be given by
	\begin{equation}
		\label{phil}
		\tilde	{\phi_{l}}_{j}(z)=\begin{cases}
			\phi_{lj}(exp_{{\xi_{l}}_{j}}(\eps_j z) \chi_r(\eps_j z) )& \textit{if} \ \ z\in B(0,r/\eps_j)\\
			0 &  \textit{otherwise}
		\end{cases} 
	\end{equation}
	
	Then, for some $c>0$, $||\tilde {\phi_l}_j||_{H^1(\re^n)}\leq ||\tilde {\phi_l}_j||_{\eps_j}\leq c$. Hence, $\tilde {\phi_l}_j$ converges weakly to some $\tilde \phi$ in $H^1(\re^n)$ and strongly in $L_{loc}^q(\re^n)$, for $q\in [2,p_n)$. Also
	
	$$\left|\frac{1}{\eps^n}\int_M\eps^2_{j} \mathbf{c} s_gZ^h_{\eps_j,{\xi_{l}}_{j}} \phi_j d\mu_g\right|\leq \eps_j^2 c_1 \left|\int_{B(0,r/\eps_j)} \psi^h(z)\chi_r(\eps_j z)\tilde{\phi_{l}}_{j}(z) |g_{{\xi_{l}}_{j}}(\eps_j z)|^{1/2}dz\right|$$
	$$=\eps_j^2 c_1 \left(\int_{\re^n} \psi^h\tilde{\phi_{l}}_{j} dz+o(1)\right)\leq \eps_j^2 c_1 \left(\int_{\re^n}( \psi^h)^2 dz\right)^{1/2}\left( \int_{\re^n}\tilde{\phi_{l}}_{j}^2 dz\right)^{1/2}+o(\eps_j^2) $$
	$$\leq c_1 c_2 \eps_j^2 ||\tilde \phi||_{L^2(\re^n)}+ o(\eps_j^2)=o(\eps_j)$$
	
	\noindent	for some constants, $c_1$, $c_2$, since $s_g$, $||\nabla U||_{L^2(\re^n)}$ and $||\tilde \phi||_{L^2(\re^n)}$ are bounded. Therefore 
	\begin{equation}
		C a_j^{hl}+o(1)=  \frac{1}{\eps_j^n}  \int_M\left[ \eps_j^2 \nabla_g  Z_{\eps_j,\bar \xi_j}^h \nabla_g  \phi_j +(\eps_j^2 \mathbf{c}s_g+1)  Z_{\eps_j,\bar \xi_j}^h \phi_j -  f'(Y_{\eps_j,\bar \xi_j}) \phi_j Z_{\eps_j,\bar \xi_j}^h \right]  d\mu_g 
	\end{equation}	
	$$ =\frac{1}{\eps_j^n}  \int_M\left[ \eps_j^2 \nabla_g  Z_{\eps_j,\bar \xi_j}^h \nabla_g  \phi_j + Z_{\eps_j,\bar \xi_j}^h \phi_j -  f'(Y_{\eps_j,\bar \xi_j}) \phi_j Z_{\eps_j,\bar \xi_j}^h \right]  d\mu_g +o(\eps_j) $$
	$$ = \int_{\re^n} \left(  \nabla  \psi^h \nabla\tilde \phi + \psi^h \tilde \phi -  f'(U+\eps_j^2V) \tilde \phi \psi^h dz\right)  +o(1)  $$
	$$ = \int_{\re^n} \left(  \nabla  \psi^h \nabla\tilde \phi + \psi^h \tilde \phi -  f'(U)  \psi^h\tilde \phi dz\right) +o(\eps_j) +o(1)=o(1) $$
	
	It follows that $a^{hl}_j\rightarrow0$ for any $h=1,2,\cdots n$ and $l=1,2,\cdots,K$. Hence
	\begin{equation}
		\label{zeta}
		||\zeta_j||_{\eps_j}\rightarrow0 \text{\ as \ } j\rightarrow \infty
	\end{equation}
	
	And then,  for $w_j=\phi_j-(\psi_j+\zeta_j)$, assumptions on $||\psi_j||_{\eps_j}$, (\ref{phi1}) and (\ref{zeta}) imply
	\begin{equation}
		\label{wj}
		||w_j||_{\eps_j}\rightarrow 1
	\end{equation}
	
	Moreover, $w_j=i^*(f'(Y_{\eps_j,\bar \xi})\phi_j)$, from (\ref{psizeta}). Hence, $w_j$ satisfies weakly, on $M$,
	
	\begin{equation}
		\label{w2}
		-\eps^2_{j}\Delta_g w_j+(\eps_j^2\mathbf{c} s_g+1)w_j=f'(Y_{\eps_j,\bar\xi_j})w_j+f'(Y_{\eps_j,\bar\xi_j})(\psi_j+\zeta_j) 
	\end{equation}
	
	Multiplying eq. (\ref{w2}) by $w_j$ and integrating over $M$,
	
	\begin{equation}
		\label{w3}
		||w_j||_{\eps_j}^2=\frac{1}{\eps_j^n}\int_M	f'(Y_{\eps_j,\bar\xi_j})w_j^2d\mu_g+\frac{1}{\eps_j^n}\int_M f'(Y_{\eps_j,\bar\xi_j})(\psi_j+\zeta_j) w_j d\mu_g
	\end{equation}
	
	By H\"older's inequality, the second term is $o(1)$:

	$$\left| \frac{1}{\eps_j^n}\int_M f'(Y_{\eps_j,\bar\xi_j})(\psi_j+\zeta_j) w_j d\mu_g \right|$$
	$$\leq \left( \frac{1}{\eps_j^n}\int_M (f'(Y_{\eps_j,\bar\xi_j}) w_j )^2 d\mu_g \right)^{1/2}\left( \frac{1}{\eps_j^n}\int_M (\psi_j+\zeta_j)^2  d\mu_g \right)^{1/2}\leq c ||w_j ||_{\eps_j} ||\psi_j+\zeta_j||_{\eps_j}=o(1)$$

	\noindent since $ ||\psi_j||_{\eps_j}\rightarrow 0$, $ ||\zeta_j||_{\eps_j}\rightarrow 0$ and $ ||w_j||_{\eps_j}\rightarrow 1$ as $j\rightarrow \infty$.  From this and eq. (\ref{w3}), eq. (\ref{convto1}) follows. 
	
	We now prove eq. (\ref{convto0}). For $l \in \{1,2, \dots, K\}$, we define $\tilde {w_{l}}_j:\re^n\rightarrow\re$,
	
	$$\tilde {w_{l}}_j(z)=w_j(\exp_{{\xi_{l}}_j}(\eps_j z))\chi_r(\exp_{{\xi_{l}}_j}(\eps_j z)), z \in \re^n$$
	
	Since  for some $c>0$, $||\tilde {w_l}_j||_{H^1(\re^n)}^2\leq ||\tilde {w}_j||_{\eps_j}^2\leq c$. Hence, up to a subsequence, $\tilde {w_l}_j$ converges weakly to some $\tilde w_l$ in $H^1(\re^n)$ and strongly in $L_{loc}^q(\re^n)$, for $q\in [2,p_n)$. 
	We claim that $\tilde w_l$ solves weakly, in $\re^n$, the problem
	
	\begin{equation}
		\label{w4}
		-\Delta \tilde w_l+\tilde w_l=f'(U)\tilde w_l
	\end{equation}
	
	Let $\varphi\in C_0(\re^n)$. Let $\varphi_j(x):= \varphi\left(\frac{exp^{-1}_{{\xi_l}_j}(x)}{\eps_j}\right)\chi_r(exp^{-1}_{{\xi_l}_j}(x))$, for $x\in B({\xi_l}_j,\eps_j R)\subset M$. For some $R>0$ big enough so that $supp \  \varphi\subset B(0,R)$ and $j$ big enough such that  $ B({\xi_l}_j,\eps_j R)\subset  B({\xi_l}_j,r)$.
	
	Now, multiplying (\ref{w2}) by $\varphi_j$ and integrating over $M$
	
	\begin{equation}
		\label{w5}
		\frac{1}{\eps_j^n}\int_M (\eps_j^2 \nabla_g w_j \nabla_g \varphi_j +(1+s_g \mathbf{c} \eps_j^2)w_j\varphi_j )d\mu_g	= \frac{1}{\eps_j^n}\int_M  f'(Y_{\eps_j,\bar\xi_j})w_j\varphi_j d\mu_g+\frac{1}{\eps_j^n}\int_M f'(Y_{\eps_j,\bar\xi_j})(\psi_j+\zeta_j) \varphi_j d\mu_g
	\end{equation}
	
	Hence, on $\re^n$, with $x=\exp_{{\xi_l}_j}(\eps_j z)$,
	
	\begin{align*}
		& \int_{B(0,R)}\left( \sum_{s,t=1}^{n}g_{{\xi_l}_j}^{st} (\eps_j z)\frac{\partial \tilde {w_l}_j}{\partial z_s}\frac{\partial \varphi}{\partial z_t}+(1+s_g \mathbf{c} \eps_j^2)\tilde {w_l}_j\varphi \right) |g_{{\xi_l}_j(\eps_j z)}|^{1/2}dz\\
		&= \int_{B(0,R)}f'\left((U(z) +\eps_j^2 V(z))\chi_r(\eps_jz)+\sum_{ i\neq l}Y\left(\frac{\exp^{-1}_{{\xi_l}_j}\exp_{{\xi_i}_j}(\eps_j z)}{\eps_j}\right) \chi_r(\exp^{-1}_{{\xi_l}_j}\exp_{{\xi_i}_j}(\eps_j z)) \right) \\
		&{w_l}_j \ \varphi\    |g_{{\xi_l}_j}(\eps_j z)|^{1/2}dz\\
		&+ \int_{B(0,R)}f'\left((U(z) +\eps_j^2 V(z))\chi_r(\eps_jz)+\sum_{i\neq l}Y\left(\frac{\exp^{-1}_{{\xi_l}_j}\exp_{{\xi_i}_j}(\eps_j z)}{\eps_j}\right) \chi_r(\exp^{-1}_{{\xi_l}_j}\exp_{{\xi_i}_j}(\eps_j z)) \right)\\
		&(\tilde\psi_j+\tilde \zeta_j) \ \varphi\   |g_{{\xi_l}_j}(\eps_j z)|^{1/2}dz
	\end{align*}
	
	\noindent where $\tilde\psi_j:= \psi_j(\exp_{{\xi_l}_j}(\eps_jz))$ and $\tilde \zeta_j(z):= \zeta_j(\exp_{{\xi_l}_j}(\eps_jz))$ for $z\in B(0,R/\eps_j)$. Also, we estimate the third term of the left-hand side to be $o(\eps_j)$:
	
	$$\int_{B(0,R)} {\bf c}s_g \eps_j^2 \tilde  {w_l}_j \varphi  |g_{{\xi_l}_j(\eps_j z)}|^{1/2}dz\leq c \eps_j^2 \int_{B(0,R)}  \tilde  {w_l}_j \varphi  |g_{{\xi_l}_j(\eps_j z)}|^{1/2}dz $$
	
	$$\leq c \eps_j^2\left( \int_{B(0,R)}  \tilde  {w_l}_j^2 dz\right)^{1/2}\left( \int_{B(0,R)} \varphi^2  |g_{{\xi_l}_j(\eps_j z)}|dz \right)^{1/2}$$ 
	
	\begin{equation}
		\label{tildew}
		\leq  c \eps_j^2 || \tilde  {w_l}_j^2 ||_{H^1(\re^n )}c_2=o(\eps_j)
	\end{equation}
	
	\noindent with $c$ an upper bound for $s_g$, $c_2$ a bound for $\int_{B(0,R)} \varphi^2  |g_{{\xi_l}_j(\eps_j z)}|dz$ and $  \tilde  {w_l}_j^2$ being bounded in ${H^1(\re^n )}$, independently of $j$.  Hence, as $\eps_j\rightarrow 0$, from eq.(\ref{tildew}) we get

	\begin{equation}
		\label{varphi}
		\int_{\re^n}\left( \sum_{s,t=1}^{n} \delta_{st}  \frac{\partial \tilde {w_l}_j}{\partial z_s}\frac{\partial \varphi}{\partial z_t}+ \tilde {w_l}_j\varphi \right)  dz= \int_{\re^n}f'\left(U(z)\right) {w_l}_j\varphi  dz
	\end{equation}  
	
	\noindent since $\tilde \psi_j$, $\tilde \zeta_j\rightarrow0$ strongly in  ${H^1(\re^n )}$. We have also used the hypothesis eq. (\ref{Ueps4}). Thus, for all $\varphi\in C_0^{\infty}(\re^n)$,  eq. (\ref{varphi})  implies that $\tilde w_l$ solves eq. (\ref{w4}) weakly in $\re^n$.
	
	We next claim that for any $h=1,2,\cdots n$, $\tilde w_l$ satisfies

	\begin{equation}
		\label{wl}
		\int_{\re^n} \left(\nabla \psi^h \nabla \tilde w_l+\psi^h\tilde w_l \right)dz=0
	\end{equation}
	
	By direct computation, since $\phi_j, \psi_j\in K^{\perp}_{\eps_j,\bar \xi_j}$, we get
	
	$$	|\langle Z^h_{\eps_j,{\xi_l}_j},w_j\rangle_{\eps_j}|=|\langle Z^h_{\eps_j,{\xi_l}_j},\phi_j-\psi_j-\zeta_j\rangle_{\eps_j}|	=|\langle Z^h_{\eps_j,{\xi_l}_j},-\zeta_j\rangle_{\eps_j}|$$
	\begin{equation}
		\label{Zh}
		\leq || Z^h_{\eps_j,{\xi_l}_j}||_{\eps_j}||\zeta_j ||_{\eps_j}=o(1)
	\end{equation}
	
	On the other hand
	
	\begin{equation}
		\label{product}
		\langle Z^h_{\eps_j,{\xi_l}_j},w_j\rangle_{\eps_j}= \frac{1}{\eps_j^n}  \int_M\left[ \eps_j^2 \nabla_g  Z_{\eps_j,  {\xi_l}_j}^h \nabla_g  w_j +(\eps_j^2 \mathbf{c}s_g+1)  Z_{\eps_j,  {\xi_l}_j}^hw_j \right]  d\mu_g  
	\end{equation}	 
	
	Meanwhile, using H\"older's inequality, we compute the second term to be $o(\eps_j)$,
	
	$$ \left|\frac{1}{\eps_j^n}  \int_M \eps_j^2 \mathbf{c}s_g   Z_{\eps_j,  {\xi_l}_j}^hw_j   d\mu_g\right|\leq c \eps_j^2  \left(\frac{1}{\eps_j^n}  \int_M  (Z_{\eps_j,  {\xi_l}_j}^h)^2 d\mu_g \right)^{1/2} \left(\frac{1}{\eps_j^n}  \int_M w_j^2 d\mu_g \right)^{1/2}  $$
	
	$$\leq c \eps_j^2 \left(   \int_{B(0,r/\eps_j}    (\psi^h(z)\chi_r(\eps_jz ))^2 |g_{{\xi_l}_j}(\eps_j z)|^{1/2} dz  \right)^{1/2}   ||w_j||_{\eps_j}\leq c \eps_j^2 \left(   \int_{\re^n}  |\nabla U|^2dz+o(1)  \right)^{1/2}   ||w_j||_{\eps_j}=o(\eps_j)$$
	
	\noindent since $||w_j||_{\eps_j} \rightarrow 1$ as $j\rightarrow\infty$. With this estimation, and rewriting eq. (\ref{product}) on $\re^n$:
	
	$$  \langle Z^h_{\eps_j,{\xi_l}_j},w_j\rangle_{\eps_j}=  \frac{1}{\eps_j^n}  \int_M\left[ \eps_j^2 \nabla_g  Z_{\eps_j,  {\xi_l}_j}^h \nabla_g  w_j +  Z_{\eps_j,  {\xi_l}_j}^hw_j \right]  d\mu_g+o(\eps_j)$$

	$$= \int_{B(0,r/\eps_j)}\left( \sum_{s,t=1}^{n}g_{{\xi_l}_j}^{st} (\eps_j z)\frac{\partial \tilde {w_l}_j}{\partial z_s}\frac{\partial }{\partial z_t}(\psi^h(z)\chi_r(\eps_j z))+ \tilde {w_l}_j\psi^h(z)\chi_r(\eps_j z) \right) |g_{{\xi_l}_j(\eps_j z)}|^{1/2}dz+o(\eps_j)$$
	
	\begin{equation}
		\label{o1}
		=\int_{\re^n} \left(\nabla \psi^h \nabla \tilde w_l+\psi^h \tilde w_l\right)dz+o(1)
	\end{equation}
	
	Hence, the claim that $\tilde w_l$ satisfies (\ref{wl}), follows from (\ref{Zh}) and (\ref{o1}). In turn,  (\ref{w4}) and (\ref{o1}) imply that $\tilde w_l=0$.

	We now estimate 	$\frac{1}{\eps_j^n}\int_M  f'(Y_{\eps_j,\bar \xi_j})w_j^2 d\mu_g$ to be $o(1)$ in order to prove (\ref{convto0}). We will partition $M$ in different subsets: small neighborhoods of each $\xi_l$, $l=1,2,\cdots K$, where we will use  $\tilde w_l=0$, and the complement of these neighborhoods, where we will use the hypothesis (\ref{Ueps4}).
	
	Let $R_j=\frac{1}{2}\min\{d_g({\xi_l}_j,{\xi_m}_j), l\neq m\}$. Let $\tilde M=\cup_{l=1}^K B_g({\xi_l}_j,R_j)$. Hence
	
	\begin{equation}
		\label{parti}
		\frac{1}{\eps_j^n}\int_M  f'(Y_{\eps_j,\bar \xi_j})w_j^2 d\mu_g= 	\frac{1}{\eps_j^n}\sum_{l=1}^K\int_{B_g({\xi_l}_j,R_j)} f'(Y_{\eps_j,\bar \xi_j})w_j^2 d\mu_g+\frac{1}{\eps_j^n} \int_{M\setminus \tilde M} f'(Y_{\eps_j,\bar \xi_j})w_j^2 d\mu_g 
	\end{equation}
	
	For the first term of the right-hand side of (\ref{parti}), for each $l$, since $\tilde w_l=0$:
	$$	\frac{1}{\eps_j^n} \int_{B_g({\xi_l}_j,R_j)} f'(Y_{\eps_j,\bar \xi_j})w_j^2 d\mu_g$$
	
	\begin{align*}
		&	=  \int_{B(0,\eps_j R_j)} f'\left((U(z) +\eps_j^2 V(z))\chi_r(\eps_jz)+\sum_{ i\neq l}^KY\left(\frac{\exp^{-1}_{{\xi_l}_j}\exp_{{\xi_i}_j}(\eps_j z)}{\eps_j}\right) \chi_r(\exp^{-1}_{{\xi_l}_j}\exp_{{\xi_i}_j}(\eps_j z)) \right)\\
		& \tilde {w_l}_j^2   |g_{{\xi_l}_j} (\eps_j z)|^{1/2}dz \\
	\end{align*}	
	$$=o(1)$$ 
	Meanwile, for the second term of the right-hand side of (\ref{parti}), by H\"older's inequality,

	$$	\frac{1}{\eps_j^n} \int_{M\setminus \tilde M} f'(Y_{\eps_j,\bar \xi_j})w_j^2 d\mu_g$$
	$$\leq\left(  \frac{1}{\eps_j^n} \int_{M\setminus \tilde M}( f'(Y_{\eps_j,\bar \xi_j}))^{n/2}d\mu_g\right)^{\frac{2}{n}}  \left( \frac{1}{\eps_j^n} \int_{M\setminus \tilde M}w_j^{\frac{2n}{n-2}} d\mu_g\right)^{\frac{n-2}{n}}$$
	$$ \leq c_1 \left(  \frac{1}{\eps_j^n} \int_{M\setminus \tilde M}\left((p-1) \sum_{l=1}^K \left(W_{\eps_j,{\xi_l}_j}+\eps_j^2 V_{\eps_j,{\xi_l}_j} \right)^{(p-2)} \right)^{n/2} d\mu_g \right)^{2/n} ||w_j||^2_{\eps_j}$$
	$$ \leq c_2 \left(  \frac{1}{\eps_j^n} \int_{M\setminus \tilde M}\left(\sum_{l=1}^K U^{p-2}\left(\frac{\exp^{-1}_{{\xi_l}_j}(\exp_{{\xi_i}_j}(\eps_j z))}{\eps_j}\right)  \chi_r^{p-2}\left(\frac{\exp^{-1}_{{\xi_l}_j}(\exp_{{\xi_i}_j}(\eps_j z))}{\eps_j}\right)\right)^{n/2}  d\mu_g \right)^{2/n}+o(\eps_j)  $$
	$$ \leq c_2 \frac{1}{\eps_j^2} \sum_{l=1}^K \left(  \int_{B_g({\xi_l}_j,r)\setminus \tilde M}  U^{\frac{n(p-2)}{2}}\left(\frac{\exp^{-1}_{{\xi_l}_j}(\exp_{{\xi_i}_j}(\eps_j z))}{\eps_j}\right)     d\mu_g \right)^{2/n}  +o(\eps_j)$$

	$$ \leq c_2 \frac{1}{\eps_j^2}e^{-(p-2)\frac{R_j}{\eps_j}} \sum_{l=1}^K \left(  \int_{B_g({\xi_l}_j,r)\setminus \tilde M}       d\mu_g \right)^{2/n}+o(\eps_j) \leq c_3 \frac{1}{\eps^2_j} e^{-(p-2)\frac{R_j}{\eps_j}} +o(\eps_j)=o(1)  $$
	
	\noindent where we used the expansion $\left(W_{\eps_j,{\xi_l}_j}+\eps_j^2 V_{\eps_j,{\xi_l}_j} \right)^{p-2}=W_{\eps_j,{\xi_l}_j}^{p-2}+o(\eps_j)$, in the third inequality. This proves  eq. (\ref{convto0}), a contradiction to eq. (\ref{convto1}).
\end{proof}

We next prove that  	$R_{\eps,\bar \xi} $ is suitably bounded.

\begin{Lemma}
	\label{Finite2}
	There are some $\rho_0>0$, $\eps_0>0$, $C>0$ and $\sigma>0$, such that for any $\rho \in (0,\rho_0)$, $\eps\in(0,\eps_0)$ and 	  $\bar \xi \in D^{K_0}_{\eps,\rho}$,  it holds
	
	\begin{equation}
		\label{Rec}
		||R_{\eps,\bar \xi}||_{\eps}\leq \left(C \eps^3 +\sum_{i\neq j}e^{-\frac{(1+\sigma)d_g(\xi_i,\xi_j)}{2\eps}}\right).
	\end{equation}
\end{Lemma}

\begin{proof} We sketch a proof, following \cite[Lemma 5.3]{Deng} and \cite[Lemma 3.3]{Dancer}. Let $Y_{\eps,\bar \xi}=\sum_{i=1}^{K}(W_{\eps,\xi_i}+\eps^2V_{\eps,\xi_i})$. Let 
	
	\begin{equation}
		\label{ie}
		X_{\eps,\bar \xi}=		-\eps^2 \Delta_g Y_{\eps,\bar \xi}+(\eps^2\mathbf{c} s_g+1) Y_{\eps,\bar \xi}
	\end{equation}
	
	\noindent this is, 		$Y_{\eps,\bar \xi}= i_{\eps}^*(X_{\eps,\bar \xi})$.

	On the other hand, using eq. (\ref{ie}) and eq. (\ref{25}), for some $c>0$,
	$$||R_{\eps,\bar\xi}||_{\eps}= ||i_{\eps}^*\left(f(Y_{\eps,\bar\xi})\right)-Y_{\eps,\bar\xi}||_{\eps}
	\leq c|f(Y_{\eps,\bar\xi})-X_{\eps,\bar\xi}|_{p',\eps}$$
	
	\noindent Hence
	\begin{equation}
		\label{Re}
		||R_{\eps,\bar\xi}||_{\eps}\leq c\left(\left|\left(\sum_{i=1}^{K}(W_{\eps,\xi_i}+\eps^2  V_{\eps,\xi_i})\right)^{p-1}-\sum_{i=1}^{K}\left(W_{\eps,\xi_i}+\eps^2  V_{\eps,\xi_i}\right)^{p-1}\right|_{p',\eps}+\left|\sum_{i=1}^{K}\left(W_{\eps,\xi_i}+\eps^2 V_{\eps,\xi_i}\right)^{p-1}-X_{\eps,\bar\xi} \right|_{p',\eps}\right)
	\end{equation}
	
	For the first term of (\ref{Re}),  arguing as in  Lemma 3.3 in \cite{Dancer}, for some $\sigma>0$,
	\begin{equation}
		\label{first}
		\left|\left(\sum_{i=1}^{K}(W_{\eps,\xi_i}+\eps^2  V_{\eps,\xi_i})\right)^{p-1}-\sum_{i=1}^{K}\left(W_{\eps,\xi_i}+\eps^2  V_{\eps,\xi_i}\right)^{p-1}\right|_{p',\eps}= o\left(\sum_{i\neq j}^{K} e^{-\frac{1+\sigma}{2}\frac{d_g(\xi_i,\xi_j)}{\eps}}\right)
	\end{equation}

	For the second term of (\ref{Re}), we claim that for some $C>0$,
	\begin{equation}
		\label{second}
		\left|\sum_{i=1}^{K}\left(W_{\eps,\xi_i}+\eps^2 V_{\eps,\xi_i}\right)^{p-1}-X_{\eps,\bar\xi} \right|_{p',\eps}\leq C\eps^3
	\end{equation}
	
	The estimate of the Lemma follows from (\ref{ie}), (\ref{first}) and (\ref{second}).

	We next prove (\ref{second}). We first note that, since $X_{\eps,\bar\xi}=\sum_{i=1}^K X_{\eps,\xi_i}$, we get
	
	$$\left|\sum_{i=1}^{K}\left(W_{\eps,\xi_i}+\eps^2 V_{\eps,\xi_i}\right)^{p-1}-X_{\eps,\bar\xi} \right|_{p',\eps}=\left|\sum_{i=1}^{K}\left(\left(W_{\eps,\xi_i}+\eps^2 V_{\eps,\xi_i}\right)^{p-1}-X_{\eps,\xi_i}\right) \right|_{p',\eps}$$  
	\begin{equation}
		\label{o3}
		\leq \sum_{i=1}^{K} \left|\left(W_{\eps,\xi_i}+\eps^2 V_{\eps,\xi_i}\right)^{p-1}- 	 X_{\eps,\xi_i}\right|_{p',\eps}
	\end{equation}


		On the other hand, let $\tilde X_{\eps,\xi}=X_{\eps,\xi}(\exp_{\xi}(z))$. Then,  for $z\in B(0,r)\subset \re^n$,
		
		$$\tilde X_{\eps,\xi}=  -\eps^2 \Delta_g( (U_{\eps}+\eps^2V_{\eps})\chi_r)+(\eps^2\mathbf c s_g+1)( U_{\eps}\chi_r+\eps^2V_{\eps}\chi_r)$$
		$$=U_{\eps}^{p-1}(z)\chi_r(z)- \eps^2U_{\eps}(z)\Delta \chi_r(z)-2 \eps^2\langle\nabla U_{\eps}(z),\nabla \chi_r(z)\rangle-\eps^4V_{\eps}(z)\Delta\chi_r(z)-2\eps^4\langle\nabla V_{\eps}(z),\nabla \chi_r(z)\rangle $$
		$$+\eps^2\left[(p-1)|U_{\eps}|^{p-2}V_{\eps}-\frac{1}{3}R_{ijkl} z_kz_l\partial_{ij}^2U_{\eps}+\frac{2}{3}R_{kssj}z_k \partial_{j}U_{\eps}-\mathbf c s_gU_{\eps} \right] \chi_r(z)$$
		
	$$			-\eps^2\left[-\frac{1}{3}R_{ijkl} z_kz_l\partial_{ij}^2U_{\eps}+\frac{2}{3}R_{kssj}z_k \partial_{j}U_{\eps}-\mathbf c s_gU_{\eps} \right] \chi_r(z)$$
	
			\begin{equation}
			-\eps^4\frac{1}{3}R_{ijkl} z_kz_l\partial_{ij}^2V_{\eps}\chi_r(z)+\eps^4 \frac{2}{3}R_{kssj}z_k \partial_{j}V_{\eps} \chi_r(z)	+\eps^4 \mathbf c s_g V_{\eps}\chi_r+o(\eps^3)
		\end{equation}

		Then, for each $\xi_i$, $i\in i=1, 2, \dots, n$, 

		
		\begin{equation}
			\label{p1}
			\left(\frac{1}{\eps^n}\int_M \left(\left(W_{\eps,\xi_i}+\eps^2 V_{\eps,\xi_i}\right)^{p-1}-X_{\eps,\xi_i}\right)^{p'} d\mu_g\right)^{\frac{1}{p'}}
		\end{equation} 
		
		$$=\left(\frac{1}{\eps^n}\int_{B(0,r)} \left(\left(U_{\eps}+\eps^2 V_{\eps}\right)^{p-1}\chi^{p-1}-X_{\eps,\xi_i}(\exp_{\xi_i}(z))\right)^{p'} |g_{\xi_i}(z)|^{\frac{1}{2}} dz\right)^{\frac{1}{p'}}$$
		
		$$\leq C_0 \left(\frac{1}{\eps^n}\int_{B(0,r)}  | U_{\eps}^{p-1}(z)(\chi_r(z)^{p-1}-\chi_r(z))|^{p'}  dz\right)^{\frac{1}{p'}} +C_0\eps^2 \left(\frac{1}{\eps^n}\int_{B(0,r)}   U_{\eps}^{p'}(z)|\Delta \chi_r(z)|^{p'}  dz\right)^{\frac{1}{p'}} $$
		
		$$+C_0\eps^4\left(\frac{1}{\eps^n}\int_{B(0,r)}   V_{\eps}^{p'}(z)|\Delta \chi_r(z)|^{p'}  dz\right)^{\frac{1}{p'}} +C_0\eps^2 \left(\frac{1}{\eps^n}\int_{B(0,r)} ( \nabla U_{\eps}(z) \nabla \chi_r(z))^{p'}  dz\right)^{\frac{1}{p'}} 		$$
		$$+C_0\eps^4 \left(\frac{1}{\eps^n}\int_{B(0,r)} ( \nabla V_{\eps}(z) \nabla \chi_r(z))^{p'}  dz\right)^{\frac{1}{p'}}+C_0\eps^2 \left(\frac{1}{\eps^n}\int_{B(0,r)}  |(p-1)U_{\eps}^{p-2}(z)V_{\eps}(z)(\chi_r(z)^{p-1}-\chi_r(z))|^{p'}  dz\right)^{\frac{1}{p'}} $$
				$$+ C_0 \eps^4	\left(\frac{1}{\eps^n}\int_{B(0,r)}( \mathbf c s_g V_{\eps}\chi_r)^{p'}dz\right)^{\frac{1}{p'}}\leq C\eps^3$$
		

		\noindent for some $C>0$.	Note that the support of  $\chi_{r/\eps}^{p-1}-\chi_{r/\eps}$ and of the derivatives of $\chi_{r/\eps}$, are contained in 		$B(0,\frac{r}{\eps})\setminus B(0,\frac{r}{2\eps})$. Hence, for some $R>0$, the terms that contain these factors are $o(e^{-R/\eps})$, since $U$, $V$ and their derivatives decay exponentially.

	\end{proof}

	We are now ready to prove Proposition \ref{Existence}.
	
	\begin{proof}(of Proposition \ref{Existence})
		To solve eq. (\ref{perp}) we look for a fixed point of the operator $T_{\eps,\xi}: H_{\eps}\cap K^{\perp}_{\eps,\bar \xi}\rightarrow  H_{\eps}\cap K^{\perp}_{\eps,\bar \xi}$, defined by
		
		$$T_{\eps,\bar \xi}(\phi) = L^{-1}_{\eps,\bar \xi}  \left( N_{\eps,\bar \xi}(\phi)+R_{\eps,\bar \xi}  \right)$$
		
		From lemma \ref{Finite1}
		
		\begin{equation}
			\label{Norm1}
			|| T_{\eps,\bar \xi}(\phi)||_{\eps} \leq c \left( ||N_{\eps,\bar \xi}(\phi)||_{\eps}+||R_{\eps,\bar \xi} ||_{\eps} \right)
		\end{equation}
		Moreover
		\begin{equation}
			\label{Norm2}
			|| T_{\eps,\bar \xi}(\phi_1)-T_{\eps,\bar \xi}(\phi_2)||_{\eps}\leq c \left( ||N_{\eps,\bar \xi}(\phi_1)||_{\eps}- ||N_{\eps,\bar \xi}(\phi_2)||_{\eps}\right)
		\end{equation}
		
		By eq. (\ref{25})
		\begin{equation}
			||N_{\eps,\bar \xi}(\phi)||_{\eps}\leq C|f(Y_{\eps,\bar \xi}+\phi)-f(Y_{\eps,\bar \xi})-f'(Y_{\eps,\bar \xi})\phi|_{p',\eps}
		\end{equation}
		
		And by the mean value Theorem there is some $\tau\in(0,1)$ such that if $||\phi_1||$ and $||\phi_2||$ are small enough,

		$$| f(Y_{\eps,\bar \xi}+\phi_1)-f(Y_{\eps,\bar \xi}+\phi_2)-f'(Y_{\eps,\bar \xi})(\phi_1-\phi_2)|_{p',\eps}$$

		$$\leq C| f'(Y_{\eps,\bar \xi}+\phi_2+\tau(\phi_1-\phi_2))-f'(Y_{\eps,\bar \xi})(\phi_1-\phi_2)|_{p',\eps}$$
		$$\leq C| f'(Y_{\eps,\bar \xi}+\phi_2+\tau(\phi_1-\phi_2))-f'(Y_{\eps,\bar \xi})|_{\frac{p}{p-2},\eps} \ |(\phi_1-\phi_2)|_{p',\eps}$$

		Now, from \cite{Dancer} section 3, eq. (3.34), for  $||\phi_1||$ and $||\phi_2||$  small enough
		$$| f'(Y_{\eps,\bar \xi}+\phi_2+\tau(\phi_1-\phi_2))-f'(Y_{\eps,\bar \xi})|_{\frac{p}{p-2},\eps} \ |(\phi_1-\phi_2)|_{p',\eps}\leq |(\phi_1-\phi_2)|_{p',\eps}$$
		
		From eq. (\ref{Norm2}), this yields
		\begin{equation}
			\label{Contraction}
			|| T_{\eps,\bar \xi}(\phi_1)-T_{\eps,\bar \xi}(\phi_2)||_{\eps}\leq  \left( ||N_{\eps,\bar \xi}(\phi_1)||_{\eps}- ||N_{\eps,\bar \xi}(\phi_2)||_{\eps}\right)\leq c|(\phi_1-\phi_2)|_{p',\eps}
		\end{equation}
		
		\noindent for $c\in(0,1)$,  for  $||\phi_1||$ and $||\phi_2||$  small enough. Hence, $T_{\eps,\bar \xi}$ is a contraction map in a small enough ball, centered at 0 in $K^{\perp}_{\eps,\bar \xi}$.
		
		Moreover, for such fixed point $\phi_{\eps,\bar \xi}$, from eq.(\ref{Norm1})

		\begin{equation}
			\label{Norm3}
			|| \phi_{\eps,\bar \xi}||_{\eps}=	|| T_{\eps,\bar \xi}(\phi)||_{\eps}\leq  c \left( ||N_{\eps,\bar \xi}(\phi)||_{\eps}+||R_{\eps,\bar \xi} ||_{\eps} \right)
		\end{equation}

		From eq. (3.35) in \cite{Dancer}, $|| N_{\eps,\bar \xi}(\phi)||_{\eps}\leq c ( || \phi_{\eps,\bar \xi}||_{\eps}^{p-1}+|| \phi_{\eps,\bar \xi}||_{\eps}^2)$, hence  $|| N_{\eps,\bar \xi}(\phi)||_{\eps}\leq c || \phi_{\eps,\bar \xi}||_{\eps}$. It follows from Lemma \ref{Finite2},

		\begin{equation}
			\label{Norm4}
			|| \phi_{\eps,\bar \xi}||_{\eps} \leq  c \left( ||N_{\eps,\bar \xi}(\phi)||_{\eps}+||R_{\eps,\bar \xi} ||_{\eps} \right)\leq c_1  || \phi_{\eps,\bar \xi}||_{\eps} +c_2\left(\eps^3 +\sum_{i\neq j } e^{-\frac{1+\sigma}{2} \frac{d_g(\xi_i,\xi_j)}{\eps}}\right)
		\end{equation}
		This implies the estimate we were looking for 
		\begin{equation}
			\label{Norm5}
			|| \phi_{\eps,\bar \xi}||_{\eps} \leq c_3\left( \eps^3 +\sum_{i\neq j } e^{-\frac{1+\sigma}{2} \frac{d_g(\xi_i,\xi_j)}{\eps}}\right)
		\end{equation}

		Finally, the fact that the map $\bar \xi\rightarrow\phi_{\eps,\bar\xi}$ is $C^1$ follows from the Implicit Function Theorem applied to the function $F(\bar \xi, \phi)=  T_{\eps,\bar \xi}(\phi)-\phi$.

	\end{proof}

 
\section{Some technical estimates}\label{TechnicalEstimates}

In this section we provide some estimates that were used in the past sections. We also prove Proposition \ref{Critical}. We begin with a technical   estimate that compares the behavior of $\exp_{\xi_i}^{-1}\exp_{\xi_j}(w)$  and $\exp_{\xi_i}^{-1}\exp_{\xi_j}$, for small $w$. This is Lemma 5.2 in \cite{Dancer}.

  \begin{Lemma}[Lemma 5.2 in \cite{Dancer}]\label{5.2} Assume  $d_g(\xi_i,\xi_j)\leq \rho$, $\rho$ small enough. Then
 	
 	\begin{equation}
 		\label{tau}
 		\tau_{\rho}(w):=\exp_{\xi_i}^{-1}\exp_{\xi_j}(w)-\exp_{\xi_i}^{-1}\exp_{\xi_j},
 	\end{equation}
 	\noindent $w\in B(0,\rho)$, satisfies,
 	
 	$$\lim_{\rho \rightarrow 0} \sup_{w\in B(0,\rho)\setminus\{0\}} \frac{|\tau_{\rho}(w)|}{w}=1$$
 	
 	In particular, setting $t_{\rho}(z):=\frac{|\tau_{\rho}(\eps z)|}{\eps}$, for any $\eta>0$, there is some $\rho_0>0$, such that for any $\rho \in (0,\rho_0)$, it holds, for $z\in B(0,\rho/\eps)$,
 	
 	\begin{equation}
 		\label{trho}
 		(1-\eta)|z|\leq |t_{\rho}(z)|\leq  (1+\eta)|z|.
 	\end{equation}
 
 \end{Lemma}
 
 We use the  lemma above to provide some useful estimates, based on the exponential decay of $U, V$ and their derivatives. 
 
  \begin{Lemma}
 	\label{5.1}
 If  $0< a\leq b$, then for any $\delta$, $0<\delta<1$, provided $\eps>0$ and $\rho>0$ are small enough, it holds
 
 $$\int_{B(0,\rho/\eps)} U^{a}\left(\frac{\exp_{\xi_i}^{-1}\exp_{\xi_j}(\eps z)}{\eps}\right)U^{b}(z)dz=o\left(e^{-(a-\delta)\frac{d_g(\xi_i,\xi_j)}{\eps}}\right), $$

  $$\int_{B(0,\rho/\eps)} V\left(\frac{\exp_{\xi_i}^{-1}\exp_{\xi_j}(\eps z)}{\eps}\right)U (z)dz=o\left(e^{-(1-\delta)\frac{d_g(\xi_i,\xi_j)}{\eps}}\right) $$

  $$\int_{B(0,\rho/\eps)} \partial_k U(z) \  U \left(\frac{\exp_{\xi_i}^{-1}\exp_{\xi_j}(\eps z)}{\eps}\right) dz=o\left(e^{-(1-\delta)\frac{d_g(\xi_i,\xi_j)}{\eps}}\right), $$
  
    $$\int_{B(0,\rho/\eps)} \partial_k U(z) \  V \left(\frac{\exp_{\xi_i}^{-1}\exp_{\xi_j}(\eps z)}{\eps}\right) dz=o\left(e^{-(1-\delta)\frac{d_g(\xi_i,\xi_j)}{\eps}}\right), $$
and  
   $$\int_{B(0,\rho/\eps)} \partial^2_{hk}U(z) \ U \left(\frac{\exp_{\xi_i}^{-1}\exp_{\xi_j}(\eps z)}{\eps}\right)dz=o\left(e^{-(1-\delta)\frac{d_g(\xi_i,\xi_j)}{\eps}}\right). $$
 \end{Lemma}
 \begin{proof}
  Recall that $\big|\frac{\exp_{\xi_i}^{-1}\exp_{\xi_j} }{\eps}\big|=\frac{d_g(\xi_i,\xi_j)}{\eps}$. For the first estimate, we use the exponential decay of $U(z)$, we have:
 	
 	 $$\int_{B(0,\rho/\eps)} U^{a}\left(\frac{\exp_{\xi_i}^{-1}\exp_{\xi_j}(\eps z)}{\eps}\right)U^{b}(z)dz\leq c \int_{\re^n} e^{-a\big| \frac{\exp_{\xi_i}^{-1}\exp_{\xi_j}(\eps z)}{\eps}\big|-b|z|}dz$$
 	
 	$$\leq c \int_{\re^n} e^{-a\big| \frac{\exp_{\xi_i}^{-1}\exp_{\xi_j} }{\eps}+t_{\rho}(z)\big|-b|z|}dz\leq c \int_{\re^n} e^{-a\big|  \eta+t_{\rho}(z)\big|-b|z|}dz=o\left(e^{-(a-\delta)\frac{d_g(\xi_i,\xi_j)}{\eps}}\right)$$
\noindent 	where we have used the function $t_{\rho}(z)$,  defined in Lemma \ref{5.2}, and its properties (eq. (\ref{trho})), provided $\rho>0$ is small enough.

The proof of the other estimates is similar, by using the exponential decay of $U$, $V$ and their derivatives. 
 \end{proof}

 \begin{Lemma}
 	\label{K3}
It holds
 $$K_3:=\sum_{\substack{i,j=1\\ i\neq j}}^K \eps^2 \frac{1}{\eps^n} \int_{M}\ \left[-\eps^2   \Delta  W_{\eps,\xi_i} +(1+\eps^2 \mathbf{c}s_g) W_{\eps,\xi_i}-|  W_{\eps,\xi_i}|^{p-1} \right]   V_{\eps,\xi_j} d\mu_g =o(\eps^4)$$
   
 \end{Lemma}

\begin{proof}
	
	Let $X_{\eps,\xi_i}:= -\eps^2   \Delta  W_{\eps,\xi_i} + W_{\eps,\xi_i}$ and $\tilde X_{\eps,\xi_i}:= X_{\eps,\xi_i}(\exp_{\xi_i}(z))$, $z \in B(0,r)$. 
	
Since $ -\eps^2   \Delta  U_{\eps} + U_{\eps}=U_{\eps}^{p-1}$ and 	by eq. (\ref{Gamma}) on Lemma \ref{expansions}, we get

	$$\tilde X_{\eps,\xi_i}(z)= U_{\eps}^{p-1}(z)\chi_r(z)-\eps^2U_{\eps}(z)\Delta\chi_r(z)-2 \eps^2\nabla U_{\eps}(z)\nabla\chi_r(z)$$
	$$+\eps^2(g^{ij}_{\xi_i}-\delta_{ij})\partial_{ij}(U_{\eps}\chi_r)-\eps^2 g^{ij}_{\xi_i}\Gamma_{ij}^k\partial_k(U_{\eps}\chi_r)$$
	
Hence, for some $c>0$ and setting $z=\eps y$,	for $i\neq j$,
	
	$$\frac{1}{\eps^n} \left|\int_{B_g(\xi_i,r)}\left( X_{\eps,\xi_i}+\eps^2 \mathbf{c}s_g  W_{\eps,\xi_i} -W_{\eps,\xi_i}^{p-1}\right) V_{\eps,\xi_j}\ d\mu_g\right|$$	
		
	$$=\frac{1}{\eps^n} \left|\int_{B(0,r)}\left(\tilde X_{\eps,\xi_i}(z)+\eps^2 \mathbf{c}s_g U_{\eps}(z)\chi_r(z)-U_{\eps}^{p-1}(z)\chi_r^{p-1}(z)\right) V_{\eps,\xi_j}(\exp_{\xi_i}(z)) \ \ |g_{\xi_i(z)}|^{\frac{1}{2}}dz\right|$$
	
		$$\leq \frac{c}{\eps^n} \int_{B(0,r)}\left|\tilde X_{\eps,\xi_i}(z)+\eps^2 \mathbf{c}s_g U_{\eps}(z)\chi_r(z)-U_{\eps}^{p-1}(z)\chi_r^{p-1}(z)\right| V_{\eps}(\exp_{\xi_j}^{-1}\exp_{\xi_i}(z)) \ dz$$

		$$=c \int_{B(0,r/\eps)}\left|\tilde X_{\eps,\xi_i}(\eps y)+\eps^2 \mathbf{c}s_g U(y)\chi_r(\eps y)-U^{p-1}(y)\chi_r^{p-1}(\eps y)\right| V\left(\frac{\exp_{\xi_j}^{-1}\exp_{\xi_i}(\eps y)}{\eps}\right) \ dy$$

		
		$$\leq c \eps^2 \int_{B(0,r/\eps)}\left| \mathbf{c}s_g U(y)\chi_{r/\eps}( y)\right| V\left(\frac{\exp_{\xi_j}^{-1}\exp_{\xi_i}(\eps y)}{\eps}\right) \ dy$$
		
		
				$$+ c \int_{B(0,r/\eps)}\left|U^{p-1}( y)\left(\chi_{r/\eps}^{p-1}(y)-\chi_{r/\eps}(y)\right)\right| V\left(\frac{\exp_{\xi_j}^{-1}\exp_{\xi_i}(\eps y)}{\eps}\right) \ dy$$
								$$+ c \int_{B(0,r/\eps)}U( y)\left|\Delta\chi_{r/\eps}(y)\right| V\left(\frac{\exp_{\xi_j}^{-1}\exp_{\xi_i}(\eps y)}{\eps}\right) \ dy$$
								
								$$+ c \int_{B(0,r/\eps)}\left\langle \nabla U( y), \nabla\chi_{r/\eps}(y)\right\rangle V\left(\frac{\exp_{\xi_j}^{-1}\exp_{\xi_i}(\eps y)}{\eps}\right) \ dy$$

		$$+ c \int_{B(0,r/\eps)}\left|(g^{ij}_{\xi_i}(\eps y)-\delta_{ij}) \partial_{ij}(U \chi_{r/\eps})(y)\right| V\left(\frac{\exp_{\xi_j}^{-1}\exp_{\xi_i}(\eps y)}{\eps}\right) \ dy$$
		
		$$+ c \eps\int_{B(0,r/\eps)}\left|g^{ij}_{\xi_i}(\eps y)\Gamma_{ij}^k(\eps y)\partial_k(U_{\eps}\chi_r)(y)\right| V\left(\frac{\exp_{\xi_j}^{-1}\exp_{\xi_i}(\eps y)}{\eps}\right) \ dy$$

		$$:=H_0+H_1+H_2+H_3+H_4+H_5.$$

	We claim that each of these terms is $o(\eps^4)$.	Note that the support of the function $\chi_{r/\eps}^{p-1}(y)-\chi_{r/\eps}$ and of the derivatives of $\chi_{r/\eps}$, are contained in 		$B(0,\frac{r}{\eps})\setminus B(0,\frac{r}{2\eps})$. Hence, for some $R>0$, $			H_1,H_2,H_3=o(e^{-R/\eps})$.	For $H_4$:
		
	$$H_4=	c \int_{B(0,r/\eps)}\left|(g^{ij}_{\xi_i}(\eps y)-\delta_{ij}) \partial_{ij}(U \chi_{r/\eps})(y)\right| V\left(\frac{\exp_{\xi_j}^{-1}\exp_{\xi_i}(\eps y)}{\eps}\right) \ dy$$
	
$$\leq 	c \int_{B(0,r/\eps)}\left|(g^{ij}_{\xi_i}(\eps y)-\delta_{ij}) \partial_{ij}(U(y)) \chi_{r/\eps}(y)\right| V\left(\frac{\exp_{\xi_j}^{-1}\exp_{\xi_i}(\eps y)}{\eps}\right) \ dy$$

$$+	c \int_{B(0,r/\eps)}\left|(g^{ij}_{\xi_i}(\eps y)-\delta_{ij})U(y) (\partial_{ij}\chi_{r/\eps}(y))\right| V\left(\frac{\exp_{\xi_j}^{-1}\exp_{\xi_i}(\eps y)}{\eps}\right) \ dy$$

		$$+	c \int_{B(0,r/\eps)}\left|(g^{ij}_{\xi_i}(\eps y)-\delta_{ij}) (\partial_{i}U)(y) (\partial_{j}\chi_{r/\eps})(y)\right| V\left(\frac{\exp_{\xi_j}^{-1}\exp_{\xi_i}(\eps y)}{\eps}\right) \ dy$$
		
		$$=o\left(\eps^2 e^{-(1-\delta)\frac{d_g(\xi_i,\xi_j)}{\eps}}\right)+o(e^{-R/\eps})$$
	
	Where we have used $(g^{ij}_{\xi_i}(\eps y)-\delta_{ij})=o(\eps^2|y|^2)$ and  Lemma \ref{5.1} to argue that  the first term is $o(\eps^2e^{-(1-\delta)\frac{d_g(\xi_i,\xi_j)}{\eps}})$. The other terms involve derivatives of $\chi_{r/\eps}$ and hence, as argued above, are $o(e^{-R/\eps})$.		 For  $H_5$:
		
			$$H_5= c \eps \int_{B(0,r/\eps)}\left|g^{ij}_{\xi_i}(\eps y)\Gamma_{ij}^k(\eps y)\partial_k(U\chi_{r\eps})(y)\right| V\left(\frac{\exp_{\xi_j}^{-1}\exp_{\xi_i}(\eps y)}{\eps}\right) \ dy$$

			$$\leq c \eps \int_{B(0,r/\eps)}\left|g^{ij}_{\xi_i}(\eps y)\Gamma_{ij}^k(\eps y)\partial_k U(y)\chi_{r\eps}(y)\right| V\left(\frac{\exp_{\xi_j}^{-1}\exp_{\xi_i}(\eps y)}{\eps}\right) \ dy$$
			
				$$+ c \eps \int_{B(0,r/\eps)}\left|g^{ij}_{\xi_i}(\eps y)\Gamma_{ij}^k(\eps y)U(y)\partial_k(\chi_{r\eps})(y)\right| V\left(\frac{\exp_{\xi_j}^{-1}\exp_{\xi_i}(\eps y)}{\eps}\right) \ dy$$
		
			$$=o\left(\eps^2 e^{-(1-\delta)\frac{d_g(\xi_i,\xi_j)}{\eps}}\right)+o(e^{-R/\eps})$$

		Where we have used Lemma \ref{5.1} for the first term, together with the estimate  $\Gamma_{ij}^k(\eps y)=\Gamma_{ij}^k(0)+o(\eps |y|)$.  And we argue, as before, that the second term is $o(e^{-R/\eps})$, as it contains derivatives of $\chi_{r/\eps}$. Finally, for  $H_0$,

	$$H_0 = c \eps^2 \int_{B(0,r/\eps)}\left| \mathbf{c}s_g U(y)\chi_{r/\eps}( y)\right| V\left(\frac{\exp_{\xi_j}^{-1}\exp_{\xi_i}(\eps y)}{\eps}\right) \ dy $$
			$$=c \eps^2 \mathbf{c} \int_{B(0,r/\eps)}\left|  s_gU(y)\chi_{r/\eps}( y)\right| V\left(\frac{\exp_{\xi_j}^{-1}\exp_{\xi_i}(\eps y)}{\eps}\right) \ dy =o\left(\eps^2 e^{-(1-\delta)\frac{d_g(\xi_i,\xi_j)}{\eps}}\right)$$
		

		The Lemma follows from the estimates of $H_0$, $H_1$, $H_2$, $H_3$,  $H_4$ and $H_5$.  
		
\end{proof}

 \begin{Lemma}
 	\label{K4}
It holds
	$$K_4:= \frac{	\eps^4}{2}\sum_{\substack{i,j=1\\ i\neq j}}^K
  \frac{1}{\eps^n}   \int_{M } \left[-\eps^2   \Delta  V_{\eps,\xi_i}   +(1+\eps^2 \mathbf{c}s_g)  V_{\eps,\xi_i}  -(p-1)|  W_{\eps,\xi_i}|^{p-2}V_{\eps,\xi_i} \right] V_{\eps,\xi_j}  d\mu_g  
  =o(\eps^4)$$
\end{Lemma}

\begin{proof}
	Since $V$ has also exponential decay, the proof is similar to  that of Lemma \ref{K3}. 
	
		Let $Q_{\eps,\xi_i}:= -\eps^2   \Delta  V_{\eps,\xi_i} + V_{\eps,\xi_i}$ and $\tilde Q_{\eps,\xi_i}:= Q_{\eps,\xi_i}(\exp_{\xi_i}(z))$, $z \in B(0,r)$. 
	
 Since $ -\eps^2   \Delta  V_{\eps} + V_{\eps}=(p-1)U_{\eps}^{p-2}V_{\eps} + R_{kl} z_k z_l  \frac{U_{\eps}'}{|z|}$, using remark \ref{Gamma} we get

	$$\tilde Q_{\eps,\xi_i}(z)= -(p-1)U_{\eps}^{p-2}(z)V_{\eps}(z)\chi_r(z) +\eps {R_{kl}}_{\xi_i} z_k z_l  \frac{U_{\eps}'}{|z|}\chi_r(z)-\eps^2V_{\eps}(z)\Delta\chi_r(z)-2 \eps^2\nabla V_{\eps}(z)\nabla\chi_r(z)$$
	$$+\eps^2(g^{ij}_{\xi_i}-\delta_{ij})\partial_{ij}(V_{\eps}\chi_r)(z)-\eps^2 g^{ij}_{\xi_i}\Gamma_{ij}^k\partial_k(V_{\eps}\chi_r)(z)$$
	
	Hence, for some $c>0$ and setting $z=\eps y$,		for $i\neq j$,
	
	$$\frac{1}{\eps^n} \left|\int_{B_g(\xi_i,r)}\left( Q_{\eps,\xi_i}+\eps^2 \mathbf{c}s_g  V_{\eps,\xi_i} -(p-1)|  W_{\eps,\xi_i}|^{p-2}V_{\eps,\xi_i}\right) V_{\eps,\xi_j}\ d\mu_g\right|$$

	$$=\frac{1}{\eps^n} \left|\int_{B(0,r)}\left(\tilde Q_{\eps,\xi_i}(z)+\eps^2 \mathbf{c}s_g V_{\eps}(z)\chi_r(z)-(p-1)U_{\eps}^{p-2}(z)V_{\eps}(z)\chi_r^{p-1}(z)\right) V_{\eps,\xi_j}(\exp_{\xi_i}(z)) \ \ |g_{\xi_i(z)}|^{\frac{1}{2}}dz\right|$$

	$$\leq \frac{c}{\eps^n} \int_{B(0,r)}\left|\tilde Q_{\eps,\xi_i}(z)+\eps^2 \mathbf{c}s_g V_{\eps}(z)\chi_r(z) -(p-1)|  U_{\eps}|^{p-2}V_{\eps}(z)\chi_r^{p-1}(z)\right| V_{\eps}(\exp_{\xi_j}^{-1}\exp_{\xi_i}(z)) \ dz$$

	$$=c \int_{B(0,r/\eps)}\left|\tilde Q_{\eps,\xi_i}(\eps y)+\eps^2 \mathbf{c}s_g V(y)\chi_r(\eps y)-(p-1)|  U(y)|^{p-2}V(y)\chi_r^{p-1}(\eps y)\right| V\left(\frac{\exp_{\xi_j}^{-1}\exp_{\xi_i}(\eps y)}{\eps}\right) \ dy$$

	
	$$\leq  c \int_{B(0,r/\eps)}\left|U^{p-2}( y)V(y)\left(\chi_{r/\eps}^{p-1}(y)-\chi_{r/\eps}(y)\right)\right| V\left(\frac{\exp_{\xi_j}^{-1}\exp_{\xi_i}(\eps y)}{\eps}\right) \ dy$$
	$$+ c \int_{B(0,r/\eps)}V( y)\left|\Delta\chi_{r/\eps}(y)\right| V\left(\frac{\exp_{\xi_j}^{-1}\exp_{\xi_i}(\eps y)}{\eps}\right) \ dy$$
	
	$$+ c \int_{B(0,r/\eps)}\left\langle \nabla V( y), \nabla\chi_{r/\eps}(y)\right\rangle V\left(\frac{\exp_{\xi_j}^{-1}\exp_{\xi_i}(\eps y)}{\eps}\right) \ dy$$

	$$+ c \int_{B(0,r/\eps)}\left|(g^{ij}_{\xi_i}(\eps y)-\delta_{ij}) \partial_{ij}(V \chi_{r/\eps})(y)\right| V\left(\frac{\exp_{\xi_j}^{-1}\exp_{\xi_i}(\eps y)}{\eps}\right) \ dy$$
	
	$$+ c \eps\int_{B(0,r/\eps)}\left|g^{ij}_{\xi_i}(\eps y)\Gamma_{ij}^k(\eps y)\partial_k(V_{\eps}\chi_r)(y)\right| V\left(\frac{\exp_{\xi_j}^{-1}\exp_{\xi_i}(\eps y)}{\eps}\right) \ dy$$
		
	$$+ c 	\eps \int_{B(0,r/\eps)} R_{kl} y_k y_l  \frac{U'(y)}{|y|}  V\left(\frac{\exp_{\xi_j}^{-1}\exp_{\xi_i}(\eps y)}{\eps}\right) \ dy$$
	
	
	$$+ c \eps^2 \int_{B(0,r/\eps)}\left| \mathbf{c}s_g V(y)\chi_{r/\eps}( y)\right| V\left(\frac{\exp_{\xi_j}^{-1}\exp_{\xi_i}(\eps y)}{\eps}\right) \ dy$$

	$$:=L_1+L_2+L_3+L_4+L_5+L_6+L_7.$$

	Note that the support of the function $\chi_{r/\eps}^{p-1}(y)-\chi_{r/\eps}$ and of the derivatives of $\chi_{r/\eps}$, are contained in 		$B(0,\frac{r}{\eps})\setminus B(0,\frac{r}{2\eps})$. Hence, for some $R>0$,
	
$$	L_1,L_2,L_3=o(e^{-R/\eps})$$

	For $L_4$:
	
	$$L_4=	c \int_{B(0,r/\eps)}\left|(g^{ij}_{\xi_i}(\eps y)-\delta_{ij}) \partial_{ij}(V \chi_{r/\eps})(y)\right| V\left(\frac{\exp_{\xi_j}^{-1}\exp_{\xi_i}(\eps y)}{\eps}\right) \ dy$$
	$$\leq 	c \int_{B(0,r/\eps)}\left|(g^{ij}_{\xi_i}(\eps y)-\delta_{ij}) \partial_{ij}(V(y)) \chi_{r/\eps}(y)\right| V\left(\frac{\exp_{\xi_j}^{-1}\exp_{\xi_i}(\eps y)}{\eps}\right) \ dy$$
	
	$$+	c \int_{B(0,r/\eps)}\left|(g^{ij}_{\xi_i}(\eps y)-\delta_{ij})V(y) (\partial_{ij}\chi_{r/\eps}(y))\right| V\left(\frac{\exp_{\xi_j}^{-1}\exp_{\xi_i}(\eps y)}{\eps}\right) \ dy$$

	$$+	c \int_{B(0,r/\eps)}\left|(g^{ij}_{\xi_i}(\eps y)-\delta_{ij}) (\partial_{i}V)(y) (\partial_{j}\chi_{r/\eps})(y)\right| V\left(\frac{\exp_{\xi_j}^{-1}\exp_{\xi_i}(\eps y)}{\eps}\right) \ dy$$
	
	$$=o\left(\eps^2 e^{-(1-\delta)\frac{d_g(\xi_i,\xi_j)}{\eps}}\right)+o(e^{-R/\eps})$$
	
	Where we have used Lemma \ref{5.1} (to argue that  the first term is $o(e^{-(1-\delta)\frac{d_g(\xi_i,\xi_j)}{\eps}})$).  And we argue as before, that the remainder of the terms are $o(e^{-R/\eps})$, since they contain derivatives of $\chi_{r/\eps}$.	For  $L_5$:
	
	$$L_5= c \eps \int_{B(0,r/\eps)}\left|g^{ij}_{\xi_i}(\eps y)\Gamma_{ij}^k(\eps y)\partial_k(V\chi_{r\eps})(y)\right| V\left(\frac{\exp_{\xi_j}^{-1}\exp_{\xi_i}(\eps y)}{\eps}\right) \ dy$$

	$$\leq c \eps \int_{B(0,r/\eps)}\left|g^{ij}_{\xi_i}(\eps y)\Gamma_{ij}^k(\eps y)\partial_k(V)(y)\chi_{r\eps}(y)\right| V\left(\frac{\exp_{\xi_j}^{-1}\exp_{\xi_i}(\eps y)}{\eps}\right) \ dy$$
	
	$$+ c \eps \int_{B(0,r/\eps)}\left|g^{ij}_{\xi_i}(\eps y)\Gamma_{ij}^k(\eps y)V(y)\partial_k(\chi_{r\eps})(y)\right| V\left(\frac{\exp_{\xi_j}^{-1}\exp_{\xi_i}(\eps y)}{\eps}\right) \ dy$$
	
	$$=o\left(\eps^2 e^{-(1-\delta)\frac{d_g(\xi_i,\xi_j)}{\eps}}\right)+o(e^{-R/\eps})$$

	Where we have used Lemma \ref{5.1} for the first term, together with the estimate  $\Gamma_{ij}^k(\eps y)=\Gamma_{ij}^k(0)+o(\eps |y|)$.  And we argue, as before, that the second term is $o(e^{-R/\eps})$, as it contains derivatives of $\chi_{r/\eps}$.	For $L_6$:
		$$	L_6=c\eps \int_{B(0,r/\eps)} R_{kl} y_k y_l  \frac{U'(y)}{|y|}  V\left(\frac{\exp_{\xi_j}^{-1}\exp_{\xi_i}(\eps y)}{\eps}\right) \ dy=o\left(\eps e^{-(1-\delta)\frac{d_g(\xi_i,\xi_j)}{\eps}}\right)$$
\noindent by	Lemma \ref{5.1}.		For  $L_7$:

	$$L_7 = c \eps^2 \int_{B(0,r/\eps)}\left| \mathbf{c}s_g V(y)\chi_{r/\eps}( y)\right| V\left(\frac{\exp_{\xi_j}^{-1}\exp_{\xi_i}(\eps y)}{\eps}\right) \ dy $$
	$$=c \eps^2 \mathbf{c} \int_{B(0,r/\eps)}\left| s_g V(y)\chi_{r/\eps}( y)\right| V\left(\frac{\exp_{\xi_j}^{-1}\exp_{\xi_i}(\eps y)}{\eps}\right) \ dy =o\left(\eps^2 e^{-(1-\delta)\frac{d_g(\xi_i,\xi_j)}{\eps}}\right)$$
	
	Where we have used  Lemma \ref{5.1}.	
	
	The Lemma follows from the estimates of $L_0$, $L_1$, $L_2$, $L_3$,  $L_4$, $L_5$ and $L_6$.  
	
\end{proof}

The following estimate will be useful.

 \begin{Lemma}
 	\label{Wq}
Let $\bar \xi=(\xi_1, \xi_2,\dots, \xi_n)\in D_{\eps, \rho}^K$. For $q>0$, 	it holds
	$$ \frac{1}{\eps^n}   \int_M \left( |\sum_{i=1}^K W_{\eps,\xi_i}|^{q}  - \sum_{i=1}^{K}  |  W_{\eps,\xi_i}|^{q}     \right) \ d\mu_g=o(\eps^2)$$

	\noindent provided $\rho>0$ is small enough.
	
\end{Lemma}

\begin{proof}
We follow the ideas of (the proof of) Lemma 3.3 in \cite{Dancer} and some estimates in (the proof of) Lemma 4.1 in the same article.

	For $0<q\leq 2$, we use the following inequality. For any $a, b>0$, there is some $c>0$ such that
	
	\begin{equation}
		\label{abq}
		\left||a+b|^{q}-a^{q}-b^{q}\right|\leq c \ a^{\frac{q}{2}} b^{\frac{q}{2}}
	\end{equation}  
Hence,  applying (\ref{abq}) repeatedly, we have,

	$$ 	 \frac{1}{\eps^n}\left(  \int_M |\sum_{i=1}^K W_{\eps,\xi_i}|^{q}  - \sum_{i=1}^{K}    |  W_{\eps,\xi_i}|^{q}     \right) d\mu_g \leq  c\frac{1}{\eps^n}  \sum_{i=1}^K  \sum_{j=i+1}^K\int_M   W_{\eps,\xi_i}^{\frac{q}{2}} W_{\eps,\xi_j}^{\frac{q}{2}}d\mu_g $$

Moreover, for $i\neq j$,
$$	\frac{1}{\eps^n}  \int_M   W_{\eps,\xi_i}^{\frac{q}{2}} W_{\eps,\xi_j}^{\frac{q}{2}}d\mu_g =\int_{B(0,r/\eps)} U^{\frac{q}{2}}(z) \chi_r^{\frac{q}{2}}(z)  U^{\frac{q}{2}}\left(\frac{\exp_{\xi_i}^{-1}(\exp_{\xi_j} (\eps z))}{\eps}\right) \chi_r^{\frac{q}{2}} (\exp_{\xi_i}^{-1}(\exp_{\xi_j} (\eps z)))  |g_{\xi_i}(\eps z)|^{\frac{1}{2}}dz$$
$$\leq c \int_{\re^n} U^{\frac{q}{2}}(z)   U^{\frac{q}{2}}\left(\frac{\exp_{\xi_i}^{-1}(\exp_{\xi_j} (\eps z))}{\eps}\right)   dz= o\left(\sum_{i\neq j}^K e^{-(\frac{q}{2}-\delta)\frac{d_g(\xi_i,\xi_j)}{\eps}}\right)$$

\noindent where the last inequality follows from Lemma \ref{5.1}.	For $q>2$, we use the following inequalities. For any $a, b>0$, there is some $c>0$ such that

\begin{equation}
	\left||a+b|^{q}-a^{q}-b^{q}-pa^{q-1}b-q ab^{q-1}\right|\leq\begin{cases}
		 c \ a^{\frac{q}{2}} b^{\frac{q}{2}}\  \  \text{if} \  2<q\leq 3\\
		  c(a^{q-2}b^2+a^2b^{q-2}) \ \  \text{if} \ \  q>3
	\end{cases} 
\end{equation}  

Note also that, from Lemma \ref{5.1}, for $q>2$,
\begin{equation}
	\label{q2}
		\frac{1}{\eps^n}  \sum_{i\neq j}^K \int_M   W_{\eps,\xi_i}^{\frac{q}{2}} W_{\eps,\xi_j}^{\frac{q}{2}}d\mu_g= o\left(\sum_{i\neq j}^K e^{-(\frac{q}{2}-\delta)\frac{d_g(\xi_i,\xi_j)}{\eps}}\right)=o\left(\sum_{i\neq j}^K e^{-(1+\delta)\frac{d_g(\xi_i,\xi_j)}{\eps}}\right)
	\end{equation}

\begin{equation}
	\label{q3}	\frac{1}{\eps^n}  \sum_{i\neq j}^K \int_M   W_{\eps,\xi_i}^{q-2} W_{\eps,\xi_j}^{2}d\mu_g =o\left(\sum_{i\neq j}^K e^{-(\min\{q-2,2\}-\delta)\frac{d_g(\xi_i,\xi_j)}{\eps}}\right)=o\left(\sum_{i\neq j}^K e^{-(1+\delta)\frac{d_g(\xi_i,\xi_j)}{\eps}}\right)
	\end{equation}

\noindent  Hence, using eq. (\ref{abq}) with $a=|\sum_{i=2}^KW_{\eps,\xi_{i}}|$ and  $b=W_{\eps,\xi_{1}}$,  and then,  $a=|\sum_{i=j+1}^K W_{\eps,\xi_{i}}|$ and  $b=W_{\eps,\xi_{j}}$, in a successive manner, we  have,

	$$ 	 \frac{1}{\eps^n}\left(  \int_M |\sum_{i=1}^K W_{\eps,\xi_i}|^{q}  - \sum_{i=1}^{K}    |  W_{\eps,\xi_i}|^{q}     \right) d\mu_g$$
	
		$$= 	 \frac{1}{\eps^n}\left(  \int_M |\sum_{i=2}^K W_{\eps,\xi_i}|^{q}  - \sum_{i=2}^{K}    |  W_{\eps,\xi_i}|^{q}     \right) d\mu_g+q\frac{1}{\eps^n} \int_M |\sum_{i=2}^K W_{\eps,\xi_i}|^{q-1}  W_{\eps,\xi_{1}}  d\mu_g+q \frac{1}{\eps^n}\int_M  W_{\eps,\xi_{1}}^{q-1} \sum_{i=2}^{K}   W_{\eps,\xi_i}   d\mu_g$$
	$$+o\left(\sum_{i\neq j}^K e^{-(1+\delta)\frac{d_g(\xi_i,\xi_j)}{\eps}}\right)$$	
	$$=q\frac{1}{\eps^n} \int_M\sum_{j=1}^{K-1} |\sum_{i=j+1}^K W_{\eps,\xi_i}|^{q-1}  W_{\eps,\xi_{j}}  d\mu_g+q \frac{1}{\eps^n}\int_M \sum_{i<j}^{K}  W_{\eps,\xi_{i}}^{q-1}   W_{\eps,\xi_j}   d\mu_g$$
	$$+o\left(\sum_{i\neq j}^K e^{-(1+\delta)\frac{d_g(\xi_i,\xi_j)}{\eps}}\right)$$

\noindent where we have used eq. (\ref{q2}) if $2<q\leq 3$ and (\ref{q3}) if $q>3$.	Now, from eq. (4.8) in \cite{Dancer}, for $q>2$,

\begin{equation}
	\label{WWW}
	\frac{1}{\eps^n} \int_M\sum_{j=1}^{K-1} |\sum_{i=j+1}^K W_{\eps,\xi_i}|^{q-1}  W_{\eps,\xi_{j}}  d\mu_g=\sum_{i=j+1}\gamma_{ij} U\left(\frac{\exp_{\xi_i}^{-1}\xi_j}{\eps}\right)+o(\eps^2)
	\end{equation}
	
	\noindent Moreover, being $\bar\xi \in D^K_{\eps,\rho}$, then $\sum_{i\neq}U\left(\frac{\exp_{\xi_i}^{-1}\xi_j}{\eps}\right)<\eps^4$. Hence, since $\gamma_{ij}$  is bounded, it follows that (\ref{WWW}) is at least $o(\eps^2)$. On the other hand,
	

$$ \frac{1}{\eps^n}\sum_{i<j}^{K}  \int_M  W_{\eps,\xi_{i}}^{q-1}   W_{\eps,\xi_j}   d\mu_g = o\left(\sum_{i< j}^K e^{-((q-1)-\delta)\frac{d_g(\xi_i,\xi_j)}{\eps}}\right)=o\left(\sum_{i< j}^K e^{-(1+\delta)\frac{d_g(\xi_i,\xi_j)}{\eps}}\right)$$
	
\end{proof}

 \begin{Lemma}
 \label{K5}
	It holds
	$$K_5:=	  \frac{ \eps^2}{\eps^n} \left[\sum_{j=1}^K  \int_M -|\sum_{i=1}^K W_{\eps,\xi_i}|^{p-1}    V_{\eps,\xi_j} d\mu_g+  \sum_{i=1}^{K} \int_{M}    |  W_{\eps,\xi_i}|^{p-1}     V_{\eps,\xi_i}\ d\mu_g +  \sum_{i\neq j}^K \int_M |  W_{\eps,\xi_i}|^{p-1} V_{\eps,\xi_j}d\mu_g \right]   =o(\eps^4)$$
\end{Lemma}

\begin{proof}
	Note that 	$K_5$ is equivalent to,
$$ 	\eps^2  \frac{1}{\eps^n}	\sum_{j=1}^K\left[  \int_M -|\sum_{i=1}^K W_{\eps,\xi_i}|^{p-1}    V_{\eps,\xi_j} d\mu_g+  \int_{M} \sum_{i=1}^{K}    |  W_{\eps,\xi_i}|^{p-1}     V_{\eps,\xi_j}\ d\mu_g \right] $$
So that to prove the Lemma, it suffices to prove that for any $j\in \{1,2, \dots, K\}$,

$$ \frac{1}{\eps^n}\int_M \left(\left|\sum_{i=1}^K W_{\eps,\xi_i}\right|^{q}   -\sum_{i=1}^{K}   |  W_{\eps,\xi_i}|^{q}    \right) V_{\eps,\xi_j} \ d\mu_g =o(\eps^2) $$

with $q=p-1$. Since 	$  V_{\eps,\xi_j}$ is bounded for each $j \in \{1,2, \dots, K\}$, and by Lemma \ref{Wq}, using $q=p-1>0$:

$$ \frac{1}{\eps^n}\int_M \left(\left|\sum_{i=1}^K W_{\eps,\xi_i}\right|^{q}   -\sum_{i=1}^{K}   |  W_{\eps,\xi_i}|^{q}    \right) V_{\eps,\xi_j} \ d\mu_g  $$
$$ \leq c \frac{1}{\eps^n}\int_M \left(\left|\sum_{i=1}^K W_{\eps,\xi_i}\right|^{q}   -\sum_{i=1}^{K}   |  W_{\eps,\xi_i}|^{q}    \right)\ d\mu_g = o(\eps^2) $$
	
\end{proof}

 \begin{Lemma}
 	\label{K6}
	It holds, 	$K_6=o(\eps^4)$, with
	$$K_6:=$$	$$\frac{\eps^4}{2}\frac{ (p-1)  }{\eps^n}   \left[    \sum_{l=1}^K  \sum_{j=1}^K\int_{M } - |\sum_{i=1}^K  W_{\eps,\xi_i}|^{p-2}  V_{\eps,\xi_j}  V_{\eps,\xi_l} d\mu_g   +     \sum_{i=1}^K   \int_{M }|  W_{\eps,\xi_i}|^{p-2}    V_{\eps,\xi_i}^2     d\mu_g
	+  \sum_{i\neq j}^K   \int_M |  W_{\eps,\xi_i}|^{p-2}  V_{\eps,\xi_j} V_{\eps,\xi_l}  d\mu_g    
	\right]$$

\end{Lemma}

\begin{proof}
		Note that 	$K_6$ is equivalent to,
		$$ \frac{\eps^4}{2}(p-1)   \frac{1}{\eps^n} \int_{M }- \left(  \left|\sum_{i=1}^K  W_{\eps,\xi_i}\right|^{p-2}  -    \sum_{i=1}^K   |  W_{\eps,\xi_i}|^{p-2} 	\right)   V_{\eps,\xi_j}  V_{\eps,\xi_l}   d\mu_g $$

	 Hence, it suffices to prove that for any $j, l \in \{1,2, \dots, K\}$,

		$$ \frac{1}{\eps^n} \int_{M } \left(  \left|\sum_{i=1}^K  W_{\eps,\xi_i}\right|^{q}  -    \sum_{i=1}^K   |  W_{\eps,\xi_i}|^{q} 	\right)   V_{\eps,\xi_j}  V_{\eps,\xi_l}   d\mu_g =o(1)
$$
	
\noindent 	with $q=p-2$.	Since 	$  V_{\eps,\xi_j}$ is uniformily bounded for each $j, l \in \{1,2, \dots, K\}$, and by Lemma \ref{Wq}, using $q=p-2>0$:
	
	$$  \frac{1}{\eps^n} \int_{M } \left(  \left|\sum_{i=1}^K  W_{\eps,\xi_i}\right|^{q}  -    \sum_{i=1}^K   |  W_{\eps,\xi_i}|^{q} 	\right)   V_{\eps,\xi_j}  V_{\eps,\xi_l}   d\mu_g \leq  c \frac{1}{\eps^n} \int_{M } \left(  \left|\sum_{i=1}^K  W_{\eps,\xi_i}\right|^{q}  -    \sum_{i=1}^K   |  W_{\eps,\xi_i}|^{q} 	\right)   d\mu_g = o(1) $$
\end{proof}

The following are estimations of the functions $Z^l_{\eps,\xi}$, defined in eq. (\ref{Z}), they were used to prove Lemma \ref{Finite1} and were proved in \cite{Dancer}.
\begin{Lemma}[Lemma 5.3 in \cite{Dancer}]
	It holds
	\begin{equation}
			\frac{\partial}{\partial y_h^j}Z^l_{\eps, \xi_i(y_i)}=	\frac{\partial}{\partial y_h^j} W_{\eps,\xi_i(y^i)}=0,\text{ if } i\neq j
	\end{equation}
		\begin{equation}
	\left\langle	Z^l_{\eps, \xi_j},	\frac{\partial}{\partial y_h^i} W_{\eps,\xi_i(y^i)}\right\rangle_{\eps}=o(1),\text{ if } i\neq j
	\end{equation}
	
	\begin{equation}
	||	\frac{\partial}{\partial y_h^i}Z^l_{\eps, \xi_i(y_i)}||_{\eps}=o\left(\frac{1}{\eps}\right),
	\end{equation}
	
	\begin{equation}
			\left\langle	Z^l_{\eps, \xi_i},Z^h_{\eps, \xi_j}\right\rangle_{\eps}=o(1), \text{ if } i\neq j
	\end{equation}

	\begin{equation}
		\label{ZZ}
	\left\langle	Z^l_{\eps, \xi_i},Z^h_{\eps, \xi_i}\right\rangle_{\eps}=c\delta_{lh}+o(1),
\end{equation}

	where $c=\int_{\re^n} \left(|\nabla\psi^l|^2+(\psi^l)^2\right)dz$ is a positive constant.
	
\end{Lemma}

Finally, we sketch a proof of Proposition \ref{Critical}. 	We argue as in Lemma 4.1 in \cite{Micheletti} and \cite{RR}, Proposition 3.1.

\begin{proof}(Of Proposition \ref{Critical})

	Let $x^{\alpha}\in B(0,r)$, for $\alpha=1,\dots K$ and let $y^{\alpha}=\exp_{\xi_{\alpha}}(x^{\alpha})$, $\bar y=(y^1,y^2,\dots, y^K)\in M^K$. Since $\bar \xi$  is a critical point of $\bar J_{\eps}$, for $\alpha=1,\dots, K$, $i=1,\dots, n$,
$$
		\frac{\partial}{\partial x_i^{\alpha}} \bar J_{\eps}(\bar y(x)) 	=0
$$
 	$$S_{\eps}(Y_{\eps,\bar y(x)}+\phi_{\eps, \bar y(x)})=	\Pi_{\eps,\xi}^{\perp}\left\{ S_{\eps}(Y_{\eps,\bar y(x)}+\phi_{\eps, \bar y(x)})\right\}+	\Pi_{\eps,\xi}\left\{ S_{\eps}(Y_{\eps,\bar y(x)}+\phi_{\eps, \bar y(x)})\right\}$$
	
	Note that $\Pi_{\eps,\xi}^{\perp}\left\{ S_{\eps}(Y_{\eps,\bar y(x)}+\phi_{\eps, \bar y(x)})\right\}=0$ by construction of $\phi_{\eps, \bar y(x)}$. Meanwhile 
	
	$$	\Pi_{\eps,\xi}\left\{ S_{\eps}(Y_{\eps,\bar y(x)}+\phi_{\eps, \bar y(x)})\right\}=\sum_{i,\alpha}C_{\eps}^{i,\alpha}Z^i_{\eps,y^{\alpha}}$$
	
\noindent	for some functions $C_{\eps}^{i,\alpha}:B(0,r)^{K}\rightarrow \re$. We now prove that  $C_{\eps}^{i,\alpha}(0)=0$ for each $i,\alpha$. Note that for each $i,\alpha$, 
	$$0=	\frac{\partial}{\partial x_i^{\alpha}} \bar J_{\eps}(\bar y(x))=\bar J_{\eps}'(Y_{\eps,\bar y(x)}+\phi_{\eps, \bar y(x)}) \left[ \frac{\partial}{\partial x_i^{\alpha}} (Y_{\eps,\bar y(x)}+\phi_{\eps, \bar y(x)})\right]$$
	
	$$=\left\langle S_{\eps}(Y_{\eps,\bar y(x)}+\phi_{\eps, \bar y(x)}), \frac{\partial}{\partial x_i^{\alpha}}\Big|_{x=0} (Y_{\eps,\bar y(x)}+\phi_{\eps, \bar y(x)}) \right \rangle_{\eps}$$
	
	\begin{equation}
		\label{Ciszero}
=\left\langle \sum_{k,\beta}C_{\eps}^{k,\beta}(0)Z^k_{\eps,y^{\beta}}, \frac{\partial}{\partial x_i^{\alpha}}\Big|_{x=0} (Y_{\eps,\bar y(x)}+\phi_{\eps, \bar y(x)}) \right \rangle_{\eps}
\end{equation}

	Now, since $\phi_{\eps, \bar \xi(y)}\in K^{\perp}_{\eps, \bar \xi(y)}$, hence $\left\langle Z^k_{\eps,y^{\beta}}, \phi_{\eps, \bar y(x)} \right \rangle=0$. Then 
	\begin{equation}
		\label{Ziszero}
		\liminf_{\eps\rightarrow0}\Big|\left\langle Z^k_{\eps,y^{\beta}}, \left(\frac{\partial}{\partial x_i^{\alpha}}\phi_{\eps, \bar y(x)}\right)\Big|_{y=0} \right \rangle_{\eps}\Big|= 	\liminf_{\eps\rightarrow0}\Big|-\left\langle  \frac{\partial}{\partial x_i^{\alpha}} Z^k_{\eps,y^{\beta}}\Big|_{y=0}, \phi_{\eps, \bar y(x)} \right \rangle_{\eps}\Big|
	\end{equation}
	$$ =	\liminf_{\eps\rightarrow0}|| \frac{\partial}{\partial x_i^{\alpha}} Z^k_{\eps,y^{\beta}}\Big|_{y=0}||_{\eps} \hspace{10pt}  || \phi_{\eps, \bar y(x)} ||_{\eps}=0$$
	
	where the last inequality follows from  eq. (\ref{barfi})  in Proposition \ref{Existence}. Now,
	
	$$\left\langle \sum_{k,\beta}C_{\eps}^{k,\beta}(0)Z^k_{\eps,y^{\beta}}, \frac{\partial}{\partial x_i^{\alpha}}\Big|_{x=0} Y_{\eps,\bar y(x)} \right \rangle_{\eps}=\left\langle \sum_{k,\beta}C_{\eps}^{k,\beta}(0)Z^k_{\eps,y^{\beta}}, \frac{\partial}{\partial x_i^{\alpha}}\Big|_{x=0} Y_{\eps, y^{\alpha}(x)} \right \rangle_{\eps}$$
	$$=\left\langle \sum_{k}C_{\eps}^{k,\alpha}(0)Z^k_{\eps,y^{\alpha}}, \frac{\partial}{\partial x_i^{\alpha}}\Big|_{x=0} Y_{\eps, y^{\alpha}(x)} \right \rangle_{\eps}+\left\langle \sum_{k,\beta\neq \alpha}C_{\eps}^{k,\beta}(0)Z^k_{\eps,y^{\beta}}, \frac{\partial}{\partial x_i^{\alpha}}\Big|_{x=0} Y_{\eps, y^{\alpha}(x)} \right \rangle_{\eps}$$
	
Note that for $\alpha\neq \beta$, by the exponential decay of $U$ and $V$,
	$$\lim_{\eps \rightarrow0} \left\langle \sum_{k,\beta\neq \alpha}C_{\eps}^{k,\beta}(0)Z^k_{\eps,y^{\beta}}, \frac{\partial}{\partial x_i^{\alpha}}\Big|_{x=0} Y_{\eps, y^{\alpha}(x)} \right \rangle_{\eps}=0$$
	
On the other hand, 
	
	$$\left\langle \sum_{k}C_{\eps}^{k,\alpha}(0)Z^k_{\eps,y^{\alpha}}, \frac{\partial}{\partial x_i^{\alpha}}\Big|_{x=0} Y_{\eps, y^{\alpha}(x)} \right\rangle_{\eps}$$
	$$=\left\langle C_{\eps}^{i,\alpha}(0)Z^i_{\eps,y^{\alpha}}, \frac{\partial}{\partial x_i^{\alpha}}\Big|_{x=0} Y_{\eps, y^{\alpha}(x)} \right \rangle_{\eps}+\left\langle \sum_{k\neq i,\alpha}C_{\eps}^{k,\alpha}(0)Z^k_{\eps,y^{\alpha}}, \frac{\partial}{\partial x_i^{\alpha}}\Big|_{x=0} Y_{\eps, y^{\alpha}(x)} \right \rangle_{\eps}$$
	
	And since $\lim_{\eps\rightarrow 0} \eps \langle Z^i_{\eps,\xi}, \frac{\partial  Y_{\eps, \xi}}{\partial x^{k}} \rangle_{\eps}=\delta_{ik} C$, then

	$$\lim_{\eps \rightarrow0} \eps \left\langle \sum_{k}C_{\eps}^{k,\alpha}(0)Z^k_{\eps,y^{\alpha}}, \frac{\partial}{\partial x_i^{\alpha}}\Big|_{x=0} Y_{\eps, y^{\alpha}(x)} \right\rangle_{\eps}=C_{\eps}^{i,\alpha}(0)C$$
	
	It follows from (\ref{Ciszero}) and (\ref{Ziszero}) that $C_{\eps}^{i,\alpha}(0)=0$.
	
\end{proof}

\end{document}